\documentclass[tbtags,reqno]{amsart}

\usepackage[tbtags]{amsmath}
\usepackage{amsthm,amsopn,amsfonts,amssymb,epsfig,anysize}
\usepackage{tikz}
\usepackage{pgfplots}
\usepackage{gensymb}
\usepackage{graphicx, tipa}
\usepackage{mathrsfs}
\usepackage{enumitem}
\usepackage[T1]{fontenc}
\usepackage{stmaryrd}

\usetikzlibrary{automata,positioning}
\usetikzlibrary{intersections,patterns,pgfplots.fillbetween}
\newcommand{\eroot}[1]{\ensuremath{\varepsilon_{#1}}}

\newcommand{\Inv}{\ensuremath{{\text{\rm Inv}}}}
\newcommand{\Fl}{\ensuremath{Fl}}
\newcommand{\G}{\mathfrak G}

\DeclareMathOperator{\upp}{up}
\DeclareMathOperator{\chainup}{top}

\DeclareMathOperator{\uppath}{uppath}
\DeclareMathOperator{\bpath}{path}

\newcount\colveccount
\newcommand*\colvec[1]{
        \global\colveccount#1
        \begin{pmatrix}
        \colvecnext
}
\def\colvecnext#1{
        #1
        \global\advance\colveccount-1
        \ifnum\colveccount>0
                \\
                \expandafter\colvecnext
        \else
                \end{pmatrix}
                \fi
              }

\newcommand{\Z}{{\mathbb Z}} %
\newcommand{\C}{{\mathbb C}} %
\newcommand{\ov}{\overline} %
\newcommand{\eS}{\widehat S} %

\newcommand{\intersect}{\cap}

\newcommand{\union}{\cup}

\newcommand{\disjunion}{\sqcup} %
\newcommand{\Disjunion}{\bigsqcup} %
\newcommand{\dunion}{\disjunion}
\newcommand{\Dunion}{\Disjunion}
\newcommand{\from}{\colon} %
\renewcommand{\to}[1][]{\overset{#1}{\rightarrow}}
\newcommand{\subgroup}{\mathrel{\leq}}

\newcommand{\suchthat}{\mid} %
\newcommand{\st}{\suchthat} %

\renewcommand{\emptyset}{{\varnothing}}

\renewcommand{\iff}{\mathrel{\Longleftrightarrow}}
\renewcommand{\subset}{\subseteq}

\usepackage[utf8x]{inputenc} %
\usepackage{amssymb,mathrsfs,amsmath}
\usepackage{tikz}
\usepackage{upgreek}

\usepackage{color}
	\definecolor{mydefi}{cmyk}{1,0,0,.5}
	\definecolor{myred}{rgb}{.7,.1,.1}
	\definecolor{myblue}{rgb}{.1,.1,.6}
	\definecolor{mygreen}{rgb}{.1,.6,.1}
        \definecolor{lightred}{rgb}{1,0.9,0.9}
        \definecolor{lightpurple}{rgb}{0.5,0.6,1}
\usepackage[urlbordercolor={1 1 1}, pdfborder={0 0 0}, bookmarks=true,
  colorlinks=true, linkcolor=myblue, citecolor=myblue,
  urlcolor=myblue, hyperfootnotes=false]{hyperref}

\usepackage[alphabetic,abbrev]{amsrefs} %
\usepackage{enumitem}
\newcommand{\de}{\emph} %

\makeatletter
\let\save@mathaccent\mathaccent
\newcommand*\if@single[3]{%
  \setbox0\hbox{${\mathaccent"0362{#1}}^H$}%
  \setbox2\hbox{${\mathaccent"0362{\kern0pt#1}}^H$}%
  \ifdim\ht0=\ht2 #3\else #2\fi
  }
\newcommand*\rel@kern[1]{\kern#1\dimexpr\macc@kerna}
\newcommand*\widebar[1]{\@ifnextchar^{{\wide@bar{#1}{0}}}{\wide@bar{#1}{1}}}
\newcommand*\wide@bar[2]{\if@single{#1}{\wide@bar@{#1}{#2}{1}}{\wide@bar@{#1}{#2}{2}}}
\newcommand*\wide@bar@[3]{%
  \begingroup
  \def\mathaccent##1##2{%
    \let\mathaccent\save@mathaccent
    \if#32 \let\macc@nucleus\first@char \fi
    \setbox\z@\hbox{$\macc@style{\macc@nucleus}_{}$}%
    \setbox\tw@\hbox{$\macc@style{\macc@nucleus}{}_{}$}%
    \dimen@\wd\tw@
    \advance\dimen@-\wd\z@
    \divide\dimen@ 3
    \@tempdima\wd\tw@
    \advance\@tempdima-\scriptspace
    \divide\@tempdima 10
    \advance\dimen@-\@tempdima
    \ifdim\dimen@>\z@ \dimen@0pt\fi
    \rel@kern{0.6}\kern-\dimen@
    \if#31
      \overline{\rel@kern{-0.6}\kern\dimen@\macc@nucleus\rel@kern{0.4}\kern\dimen@}%
      \advance\dimen@0.4\dimexpr\macc@kerna
      \let\final@kern#2%
      \ifdim\dimen@<\z@ \let\final@kern1\fi
      \if\final@kern1 \kern-\dimen@\fi
    \else
      \overline{\rel@kern{-0.6}\kern\dimen@#1}%
    \fi
  }%
  \macc@depth\@ne
  \let\math@bgroup\@empty \let\math@egroup\macc@set@skewchar
  \mathsurround\z@ \frozen@everymath{\mathgroup\macc@group\relax}%
  \macc@set@skewchar\relax
  \let\mathaccentV\macc@nested@a
  \if#31
    \macc@nested@a\relax111{#1}%
  \else
    \def\gobble@till@marker##1\endmarker{}%
    \futurelet\first@char\gobble@till@marker#1\endmarker
    \ifcat\noexpand\first@char A\else
      \def\first@char{}%
    \fi
    \macc@nested@a\relax111{\first@char}%
  \fi
  \endgroup
}
\makeatother

\newtheorem{thm}{Theorem}
\newtheorem*{thm*}{Theorem}
\newtheorem{lem}[thm]{Lemma}
\newtheorem*{lem*}{Lemma}
\newtheorem{prop}[thm]{Proposition}
\newtheorem*{prop*}{Proposition}
\newtheorem{cor}[thm]{Corollary}
\newtheorem*{cor*}{Corollary}
\newtheorem{conj}[thm]{Conjecture}
\newtheorem*{conj*}{Conjecture}

\theoremstyle{definition}
\newtheorem{defn}[thm]{Definition}
\newtheorem{definition}[thm]{Definition}
\newtheorem*{defn*}{Definition}
\newtheorem{example}[thm]{Example}
\newtheorem*{example*}{Example}

\newtheorem*{examples*}{Examples}

\newtheorem*{alg*}{Algorithm}
\newtheorem{rmk}[thm]{Remark}
\newtheorem*{rmk*}{Remark}

\newtheorem*{rmks*}{Remarks}

\numberwithin{equation}{section}
\allowdisplaybreaks

\usepackage{caption}
\usepackage{subcaption}
\usepackage{todonotes}
\usepackage{ytableau}
\usepackage{xcolor}
\usepackage{mathtools}
\usepackage{pgf}
\usepackage{tikz}
\usepackage{tikz-cd}
\usepackage{graphicx}
\numberwithin{thm}{section}

\newcommand{\sym}{\Lambda}

\makeatletter
\newcommand\mathcircled[1]{%
  \mathpalette\@mathcircled{#1}%
}
\newcommand\@mathcircled[2]{%
  \tikz[baseline=(math.base)] \node[draw,circle,inner sep=1pt] (math) {$\m@th#1#2$};%
}
\makeatother

\newcommand{\w}{\mathfrak{w}}

\newcommand{\Par}{\operatorname{Par}}
\newcommand{\lowers}{\mathcal{L}}
\newcommand{\g}{\mathfrak{g}}
\newcommand{\tg}{\mathfrak{\tilde g}}
\DeclareMathOperator{\up}{up}
\DeclareMathOperator{\down}{down}
\renewcommand{\path}{\operatorname{path}}
\renewcommand{\top}{\operatorname{top}}

\newcommand{\rootconcat}{\uplus}

\newcommand{\ZZ}{\mathbb Z}
\newcommand{\CC}{\mathcal C}
\newcommand{\core}{\mathfrak{c}}
\newcommand{\partition}{\mathfrak{p}}

\newcommand{\Gr}{\mathrm{Gr}}

\DeclareMathOperator{\supp}{supp}

\newtheorem*{problem*}{Problem}
\newtheorem{conjecture}[thm]{Conjecture}
\newtheorem*{conjecture*}{Conjecture}

\begin{document}
\title{$K$-theoretic Catalan functions}

\author{Jonah Blasiak}
\address{Department of Mathematics, Drexel University, Philadelphia, PA 19104}
\email{jblasiak@gmail.com}

\author{Jennifer Morse}
\address{Department of Math, University of Virginia, Charlottesville, VA 22904}
\email{morsej@virginia.edu}

\author{George H. Seelinger}
\address{Department of Math, University of Virginia, Charlottesville, VA 22904}
\email{ghs9ae@virginia.edu}

\thanks{Authors were supported by NSF Grants DMS-1855784 (J.~B.)
and DMS-1855804 (G.~S. and J.~M.).}

\date{}
\maketitle

\begin{abstract}
We prove that the $K$-$k$-Schur functions are part of a family of inhomogenous symmetric functions
whose top homogeneous components are Catalan functions, the Euler characteristics of
certain vector bundles on the flag variety.  Lam-Schilling-Shimozono
identified the $K$-$k$-Schur functions
as Schubert representatives for $K$-homology of the affine Grassmannian for SL$_{k+1}$.
Our perspective reveals that
the $K$-$k$-Schur functions satisfy a shift invariance property, and we deduce
positivity of their branching coefficients from a positivity
result of Baldwin and Kumar. We further show that a slight adjustment of our formulation for $K$-$k$-Schur functions produces a
second shift-invariant basis which conjecturally has both positive branching
and a rectangle factorization property.
Building on work of Ikeda-Iwao-Maeno, we conjecture that this second basis gives the images of the Lenart-Maeno quantum Grothendieck polynomials under a $K$-theoretic analog of the Peterson isomorphism.
\end{abstract}

 \section{Introduction}

Ungraded $k$-Schur functions from~\cite{LMktab} form a combinatorially defined
basis for a sub-Hopf algebra $\Lambda_{(k)}$ of symmetric functions
that satisfies
many beautiful positivity properties.  Geometrically, they are Schubert representatives
for the homology of the affine Grassmannian ${\rm Gr} = G(\mathbb C((t))/G(\mathbb C[[t]])$
of $G={\rm SL}_{k+1}$~\cite{Lam08}.  Under the Peterson isomorphism~\cite{LStoda},
they are images of the quantum Schubert polynomials
constructed by Fomin, Gelfand, and Postnikov~\cite{fomin-gelfand-postnikov}.
Hence the $k$-Schur structure constants are Gromov-Witten invariants for the quantum cohomology ring of the complete flag variety ${\rm Fl}_{k+1}$.

Over the last several decades, a $K$-theoretic counterpart to this story has been emerging.
The \mbox{$K$-homology} $K_*({\rm Gr})$ is also Hopf isomorphic
to $\Lambda_{(k)}$~\cite{lss}, and Schubert representatives are now
given by a basis of inhomogeneous symmetric functions
called {\it $K$-$k$-Schur functions}, $g_\lambda^{(k)}\in \Lambda_{(k)}$.
They satisfy an elegant Pieri rule and are conjecturally surrounded with positivity
properties.
Foremost is the following branching property.

\begin{conjecture}[\cite{lss}*{Conjecture 7.20(3)}, \cite{morse}*{Conjecture 44}]
\label{branchingconj}
For any partition $\lambda$ with $\lambda_1\leq k$,
\begin{equation}
\label{eq:gbranching}
  g_\lambda^{(k)} = \sum_\mu a_{\lambda \mu} \,g_{\mu}^{(k+1)}
\quad\text{
satisfies \((-1)^{|\lambda|-|\mu|} a_{\lambda \mu} \in \Z_{\geq 0}\).
}
\end{equation}
\end{conjecture}

Proofs for positivity results have not been accessible from the previous geometric and algebraic descriptions
 of $K$-$k$-Schur functions.
We overcome this with an explicit raising operator formula for $g_\lambda^{(k)}$ which enables us to settle Conjecture~\ref{branchingconj}
and to derive new properties of the basis.

We prove this formula by connecting it to the Pieri rule for $g_\lambda^{(k)}$
through careful analysis of intermediate raising operator objects
between  $g_\lambda^{(k)}$ and  $g_{1^r} g_\lambda^{(k)}$.
This powerful approach to
Schubert calculus was initiated
in \cite{tamvakis,bkt1,bkt2}, further leveraged in~\cite{anderson,anderson-fulton}.
We advance this program
using methods of~\cite{catalans}, which came out of the study of Euler 
characteristics of vector bundles
on the flag variety \cite{BroerNormality, SW, CH, pany}.
Therein,
the $k$-Schur basis is identified with a subfamily of symmetric functions called
Catalan functions.  These functions are defined by a raising operator formula and
are indexed by pairs \((\Psi,\gamma)\), where $\Psi$ is one of Catalan many
upper order ideals in the set of positive \(A_{\ell-1}\) roots, $\Delta^+_\ell$,
and \(\gamma \in \Z^\ell\).

We extend the Catalan functions to an inhomogenous family of symmetric functions
using additional information from a multiset $M$ supported on $\{1,\ldots,\ell\}$.
These functions, $K(\Psi,M,\gamma),$ are called Katalan
functions.
Computer experimentation leads us to propose natural conditions for
Schur positive expansions, as well as positive expansions (up to
predictable sign) in the basis of dual stable Grothendieck polynomials $\{g_\lambda\}$, Hall-dual to the basis of Fomin-Kirillov stable
Grothendieck polynomials $\{G_\mu\}$~\cite{FK,FK2,LasG}.

We prove that the $K$-$k$-Schur functions are a distinguished subfamily of Katalan functions.
The simplicity of our formula reveals that the $K$-$k$-Schur basis satisfies {\it shift invariance}:
\begin{equation}
\label{shiftinvariance}
  G_{1^\ell}^\perp\, g_{\lambda+1^\ell}^{(k+1)} = g_\lambda^{(k)}\,.
\end{equation}
This remarkable property implies that the branching coeffients of~\eqref{eq:gbranching} are
none other than a subset of dual Pieri coefficients.  From this, a positivity result of Baldwin and
Kumar~\cite{baldwin-kumar} enables us to prove several conjectures about $K$-$k$-Schur functions,
including positive branching.

Another application of the Katalan formulation for $K$-$k$-Schur functions
involves the quantum $K$-theory ring, a deformation of the Grothendieck ring of coherent sheaves
on $\Fl_n$ studied by Givental and Lee~\cite{givental-lee}.
Kirillov and Maeno~\cite{KM} proposed a presentation  $\mathcal{QK}(\Fl_{n})$
for ${QK}(Fl_{n})$ which was recently established by Anderson-Chen-Tseng~\cite{act}.
Lenart and Maeno~\cite{LenartMaeno} introduced {\it quantum Grothendieck polynomials}
$\mathfrak G_w^Q$ as potential representatives for the Schubert basis,
just confirmed in~\cite{lns}.

Using Ruijsenaars’s relativistic Toda lattice,
Ikeda-Iwao-Maeno~\cite{IIM} %
produced an explicit ring isomorphism $\Phi$ between localizations of $K_*(\Gr)$ and
$\mathcal{QK}(\Fl_{k+1})$
and conjectured that the images of quantum Grothendieck polynomials
expand unitriangularly into $K$-$k$-Schur functions with coefficients having predictable sign;
building on this work, Ikeda conjectured a precise description for the
images.
Kato~\cite{kato} also considers related ideas in general type.

\begin{conj}[\cite{IIM}*{Conjecture 1.8}, \cite{ikedaprivate}]
\label{eq:IIMconj}
For $w\in S_{k+1}$,
\begin{align}
\label{gwiggle}
  \Phi(\G_w^Q)  = \frac{\tilde {g}_w}{\prod_{d \in \text{\rm Des}(w)} g_{{(k+1-d)^d}}}\,,
  \qquad  \text{ for \  }
\tilde g_w := (1-G_1^\perp) \bigg(\sum_{\mu_1 \le k,\,
w_\mu\leq w_\lambda } g_\mu^{(k)}
\bigg) \in\Lambda_{(k)}\,,
\end{align}
where $\lambda=\theta(w)^{\omega_k}$ is a partition with  $\lambda_1\le k$, defined in \S\ref{ss Katalan for quantum},
$w_\lambda$ denotes  the minimal coset representative
of \(S_{k+1}\) in  $\eS_{k+1}$ associated to  $\lambda$ (see \S\ref{ss raising op KkSchur}), and $\le$ denotes Bruhat order
on \(\eS_{k+1}\).
\end{conj}
To give geometric context for this conjecture, under the Hopf algebra isomorphism
$\sym_{(k)} \to K_*({\rm Gr})$,
the sum \(\sum_{\mu_1 \le k,\,w_\mu\leq w_\lambda } g_\mu^{(k)}\) maps to the class of the structure
sheaf of the Schubert variety \(X_{w_\lambda} \subset \Gr\),
whereas \(g^{(k)}_\lambda\) maps to the
class of the ideal sheaf of the boundary \(\partial X_{w_\lambda}\);
see \cite{llms-conj-peterson-isom}*{Theorem 1} and \cite{lss}*{Theorems 5.4 and 7.17(1)}.

We conjecture an explicit operator formula for the  $\tilde{g}_w$'s by realizing them as a
subfamily of Katalan functions; it requires only a slight adjustment to our  Katalan description of $K$-$k$-Schur functions.

We are also able to
verify Conjecture~\ref{eq:IIMconj} for Grassmannian permutations,
completing the proof strategy of \cite{IIM}, by establishing the following missing ingredient,
which is an immediate consequence of the Katalan formulation for $K$-$k$-Schur functions.

\begin{conj}[\cite{morse}]
For a partition $\lambda$ where $\lambda_1+\ell(\lambda)-1\leq k$, $g_\lambda^{(k)}=g_\lambda\,.$
\end{conj}

\section*{Acknowledgements}
We thank Takeshi Ikeda for generously sharing his ideas building on the work of Ikeda-Iwao-Maeno.  This inspired and enabled us to check Conjecture~\ref{c tg}.
We also thank Mark Shimozono for pointing out the work of Baldwin and Kumar~\cite{baldwin-kumar} as well as the reference \cite{llms-conj-peterson-isom}.  This research was supported by computer exploration, using the open source mathematical system SageMath~\cite{sage}.

\section{Main Results}\label{section 2}

We work in the ring $\Lambda =\ZZ[e_1,e_2,\ldots] =  \ZZ[h_1,h_2,\dots]$ of symmetric functions
in infinitely many variables $\mathbf{x} = (x_1, x_2, \dots)$, where
$e_d = e_d(\mathbf{x})= \sum_{i_1<\cdots < i_d} x_{i_1}\cdots x_{i_d}$
and $h_d = \sum_{i_1\leq\cdots \leq i_d} x_{i_1}\cdots x_{i_d}$.
Set $h_0 = 1$  and $h_d = 0$ for  $d < 0$ by convention.
For $\gamma\in\mathbb Z^\ell$, define $h_\gamma = h_{\gamma_1}\cdots h_{\gamma_\ell}$
and define Schur functions,
\begin{align}
\label{eq:s-gamma}
s_\gamma = \det( h_{\gamma_i + j-i} )_{1 \le i, j \le \ell}\,.
\end{align}

Fix $k\in\mathbb Z_{>0}$ and $\ell\in\mathbb Z_{\geq 0}$ throughout.
Set $\sym_{(k)}= \ZZ[h_1,\dots, h_k] \subset \sym$.
Let \(\Par^k_\ell = \{(\mu_1,\ldots,\mu_\ell) \in \Z^\ell \st k \geq
\mu_1 \geq \cdots \geq \mu_\ell \geq 0\}\) denote the set of
partitions contained in the \(\ell \times k\) rectangle and let
\(\Par^k\) be the set of partitions \(\mu\) with \(\mu_1 \leq k\).
The \de{length} \(\ell(\mu)\) is always the number of nonzero parts of \(\mu\).

\subsection{Katalan functions: definition and first properties}

This work builds off previous studies of symmetric functions known as Catalan functions,
introduced in \cite{CH, pany} and studied further in \cite{catalans,ksplitcatalans}.
Catalan functions involve a parameter $t$,
but we will only work with their $t=1$ specialization as this is necessary for
applications to affine Schubert calculus.
We define Catalan functions from a description in \cite[Proposition~4.7]{catalans}.
Consider the set of labels
$\Delta^+_\ell= \Delta^+ := \big\{(i,j) \mid 1 \le i < j \le \ell \big\}$
for the positive roots of $A_{\ell-1}$. %
A \emph{root ideal}  $\Psi$ is an upper order ideal of the poset $\Delta^+$ with partial order
given by $(a,b) \leq (c,d)$ when $a\geq c$ and $b\leq d$.
The complement $\Delta^+ \setminus \Psi$ is a lower order ideal of  $\Delta^+$.
A {\it Catalan function}, indexed by a pair $(\Psi, \gamma)$ consisting of a root ideal
$\Psi$ and a weight $\gamma \in \ZZ^{\ell}$, is defined by
\begin{align}
\label{eq:H-definition-CHL}
H(\Psi;\gamma)
&= \prod_{(i,j) \in \Delta^+ \setminus \Psi} (1-{R}_{i j})h_\gamma\,,
\end{align}
where the raising operator $R_{i j}$ acts on subscripts
by  $R_{i j} h_\gamma = h_{\gamma + \epsilon_{i} - \epsilon_{j}}$
and \(\epsilon_i\) is the unit vector with a \(1\) in position \(i\) and \(0\)'s elsewhere.
Below we also use raising operators on other elements indexed by weights in  $\Z^\ell$.
Raising operators were introduced by Young \cite{young} and formalized rigourously  by Garsia-Remmel \cite{garsia-remmel1,garsia-remmel2}.
Their standard usage is somewhat informal; they will be treated formally in~Section~\ref{section 3} (in a different way from Garsia-Remmel).

Our work requires the following inhomogeneous version of the $h_m$'s.
For \(m,r \in \Z\), define
\[
      k_m^{(r)} = \sum_{i=0}^m \binom{r+i-1}{i} h_{m-i}\,,
    \]
where  $\binom{n}{i} = \frac{n(n-1)\cdots (n-i+1)}{i!}$  and  $\binom{n}{0} = 1$ for $n\in \Z, i\in \Z_{\ge 1}$;
thus note that \(k_m^{(0)} = h_m\) and \(k_{m}^{(r)} = 0\) when $m<0$.
For \(\gamma \in \Z^\ell\), let
$g_\gamma = \,\det(k_{\gamma_i+j-i}^{(i-1)})_{1 \leq i,j \leq \ell}$.
When $\gamma$ is a partition, these are the {\it dual stable Grothendieck polynomials},
first studied implicitly in~\cite{lenart} and determinantally
formulated in~\cite{lascouxnaruse}.
We use an alternative characterization, proved in Section~\ref{proof-of-dual-grothendieck-jacobi-trudi} of the Appendix:
\begin{equation}
  \label{eq:dual-grothendieck-jacobi-trudi}
    g_\gamma \,=\,\prod_{1\leq i < j\leq\ell} (1-R_{ij}) k_\gamma\,,
\;\,\text{
where
\(
  k_\gamma := k_{\gamma_1}^{(0)} k_{\gamma_2}^{(1)} \cdots k_{\gamma_\ell}^{(\ell-1)}\,.
\)
}
\end{equation}

\begin{defn}\label{def:Kat-g}
  For a root ideal \(\Psi \subset \Delta^+_\ell\), a multiset \(M\)
  with \(\supp(M) \subset \{1,\ldots,\ell\}\), and \(\gamma \in \Z^\ell\), we define the
  \de{Katalan function}
  \begin{equation}
\label{eq:Kat-g}
    K(\Psi; M; \gamma) := \prod_{j \in M} (1-L_j)
    \prod_{(i,j) \in \Psi} (1-R_{ij})^{-1} g_\gamma\,,
\end{equation}
  where %
the \de{lowering operator}  $L_j$ acts on the subscripts of  $g_\gamma\in \Lambda$ by \(L_j g_\gamma=
g_{\gamma- \epsilon_j}\).
\end{defn}

The following alternative formulation  gives additional insight (see \eqref{eq:formal-def-of-Kat}--\eqref{eq:formal2} for the proof).

\begin{prop}
\label{def:Kat}
  For a root ideal \(\Psi \subset \Delta^+_\ell\), a multiset \(M\)
  with \(\supp(M) \subset \{1,\ldots,\ell\}\), and \(\gamma\in
  \Z^\ell\), %
 \[
    K(\Psi; M; \gamma) = \prod_{j \in M} (1-L_j)
    \prod_{(i,j) \in \Delta^+ \setminus \Psi} (1-R_{ij})\, k_\gamma\,.
  \]
\end{prop}

The family of Katalan functions contains several well-studied symmetric
function bases.

\begin{prop}\label{prop:extremal}
  Let \(\gamma \in \Z^\ell\).
  \begin{enumerate}
  \item The Katalan functions contain the family of Catalan funcitons:
    \(K(\Psi;\Delta^+_\ell;\gamma) = H(\Psi;\gamma)\)
    for any root ideal \(\Psi \subset \Delta^+_\ell\).
     In particular,  \(K(\emptyset;\Delta^+_\ell;\gamma) =
    s_\gamma\) and \(K(\Delta^+_\ell;\Delta^+_\ell;\gamma) = h_\gamma\).
  \item
  \(K(\emptyset;\emptyset;\gamma) = g_\gamma\).
  \item \(K(\Delta^+_\ell;\emptyset;\gamma) = k_\gamma\).
    \end{enumerate}
\end{prop}
\begin{proof}
  Statement (b) is immediate from Definition~\ref{def:Kat-g} and (c) is immediate from
  Proposition~\ref{def:Kat}. To prove (a), for \(m,r \in \Z\), we note that, by Pascal's formula,
\begin{equation}
\label{pascal-kh-identity}
    k_{m-1}^{(r)} + k_m^{(r-1)}
    = \sum_{i=0}^{m} \left[\binom{r+i-2}{i-1} +
    \binom{r+i-2}{i}\right] h_{m-i}
    =
    \sum_{i=0}^m
    \binom{r+i-1}{i} h_{m-i} = k_m^{(r)}\,.
\end{equation}
Therefore,
\(\prod_{(i,j) \in \Delta^+} (1-L_j) k_\gamma= h_\gamma\) and thus (a) follows from Proposition~\ref{def:Kat} and \eqref{eq:H-definition-CHL}.
\end{proof}

Although Katalan functions
are defined for arbitrary multisets, we mainly work with those where the associated multiset comes from a root
ideal \(\lowers \subset \Delta^+_\ell\) via the function
\begin{equation}
\label{eq:root-ideal-to-multiset}
L(\lowers) = \Dunion_{(i,j) \in \lowers} \{j\} \,.
\end{equation}
In this scenario, we use the shorthand \(K(\Psi;\lowers;\gamma) = K(\Psi;L(\lowers);\gamma)\).

\subsection{A raising operator formula for  $K$-$k$-Schur functions}
\label{ss raising op KkSchur}

In \cite{catalans}, the \de{$k$-Schur functions} $\{s_\mu^{(k)}\}_{\mu \in \Par^k}$ were identified with
a subfamily of Catalan functions, namely $s_\mu^{(k)}=H(\Delta^{k}(\mu);\mu)$ where
\begin{align}\label{e Deltak def}
\Delta^{k}(\mu) = \{(i,j) \in \Delta^+_{\ell} \mid  k-\mu_i + i < j\}\,.
\end{align}

\begin{defn}\label{def:KatKks}
For \(\lambda \in \Par_\ell^k\), define the \de{$k$-Schur Katalan function} by
\[ \g_\lambda^{(k)} = K(\Delta^{k}(\lambda); \Delta^{k+1}(\lambda); \lambda) \,.
  \]
\end{defn}

We show that the $k$-Schur Katalan functions are the $K$-$k$-Schur functions. %
This operator formula is considerably more direct and explicit than any previously known description of the $K$-$k$-Schur functions
and  readily resolves
several outstanding conjectures, including positive branching.

The $K$-$k$-Schur functions are
defined using the \de{affine symmetric group} \(\eS_{k+1}\),
the group with generators \(\{s_i \st i \in I\}\) for \(I = \{0,\ldots,k\}\) subject
to the relations \(s_i^2 = id,\, s_i s_{i+1} s_i = s_{i+1} s_i s_{i+1},
\,s_i s_j = s_j s_i\) for \(i-j \not\equiv 0,\pm 1\), with all indices
considered modulo \(k+1\).  The \de{length}  $\ell(w)$ of $w \in \eS_{k+1}$
is the minimum $m$ such that $w=s_{i_1}s_{i_2}\cdots s_{i_m}$ for some $i_j\in I$;
any expression for $w$ with $\ell(w)$ generators is said to be \de{reduced}.
The set of \de{affine Grassmannian elements} \(\eS_{k+1}^0\) are
the  minimal length coset representatives of
\(S_{k+1}\) in \(\eS_{k+1}\), where
\(S_{k+1} = \langle s_1, \ldots, s_{k} \rangle \subgroup \eS_{k+1}\).
There is a bijection
\(\w \from \Par^k \to \eS_{k+1}^0,\) given by  \(\lambda \mapsto w_\lambda \) for
$w_\lambda = (s_{\lambda_\ell-\ell} \cdots s_{-\ell+1}) \cdots (s_{\lambda_2-2} \cdots s_{-1})(s_{\lambda_1-1} \cdots s_0)$  where $\ell = \ell(\lambda)$ (see \cite[\S8.2]{LMtaboncores}).
For example, for  $k=3$, $w_{3221} = s_1 \, s_3s_2 \, s_0s_3 \, s_2 s_1s_0$.

The \emph{0-Hecke algebra}  $H_{k+1}$ is the free  $\ZZ$-algebra generated by \(\{T_i \st i \in I\}\)
with the same relations as $\eS_{k+1}$ except  $T_i^2 = -T_i$ in place of  $s_i^2 = id$.
It has a  $\ZZ$-basis  $\{T_w \mid w \in \eS_{k+1}\}$, where
$T_w = T_{i_1} T_{i_2}\cdots T_{i_m}$ for any reduced expression
$w=s_{i_1}s_{i_2}\cdots s_{i_m}$.

The following descriptions of the $K$-$k$-Schur functions  $g^{(k)}_\lambda$
are implicit in~\cite{lss, morse} and are verified in Section~\ref{sec:eq-of-Kks}
of the Appendix.
An element $w\in \eS_{k+1}$ is  \de{cyclically increasing}
if it can be written as $w=s_{i_1}s_{i_2}\cdots s_{i_m}$, for distinct indices  $i_j$
such that an index $i$ never occurs to the east of an ${i+1}$ (modulo \(k+1\)).

\begin{thm}
\label{t K k Schur basics}
There is a Hopf algebra isomorphism $\Theta \from K_*({\rm Gr_{SL_{k+1}}}) \to \sym_{(k)}$;
the $K$-homology Schubert basis element $\xi^0_{w_\lambda}$ has image denoted
$g_\lambda^{(k)} = \Theta(\xi^0_{w_\lambda})$, for  $\lambda \in \Par^k$.
The $\{g^{(k)}_\lambda \}_{\lambda\in \Par^k}$ form a basis for
 $\sym_{(k)}$ and satisfy the following Pieri rule for all  \(r \in [k]\)\,:
\begin{equation}
\label{gpieri}
g_{1^r} g_\lambda^{(k)} = \sum_{\substack{u \in \eS_{k+1} \text{\,cyclically increasing}
\\ \ell(u) = r  \\ T_u T_{w_\lambda} = \pm T_{w}; \, w \in \eS_{k+1}^0 }}
(-1)^{\ell(w_\lambda)+r-\ell(w)} g_{\w^{-1}(w)}^{(k)}\,.
\end{equation}
Moreover, the $\{g^{(k)}_\lambda \}_{\lambda\in \Par^k}$ are the unique elements of  $\sym_{(k)}$
satisfying \eqref{gpieri} for all  \(r \in [k]\).
\end{thm}

We will show in Theorem~\ref{thm:katalan-kschur-pieri} that the $k$-Schur Katalan functions
$\g_\lambda^{(k)}$ satisfy~\eqref{gpieri}, establishing

\begin{thm}\label{formulations-are-equal}
For any $\lambda\in\Par^k$,
  \(\g_\lambda^{(k)} = g_\lambda^{(k)}\).
Thus, the $k$-Schur Katalan functions are representatives for the Schubert basis of the K-homology of the affine Grassmannian of
$SL_{k+1}$.
\end{thm}

\subsection{Positive branching}

The foremost application of the Katalan function formulation
for $K$-$k$-Schur functions is the ease with which shift invariance~\eqref{shiftinvariance}
follows; we model developments in \cite{catalans} where it was shown that
$k$-Schur functions satisfy a similar shift invariance property,
$e_{\ell}^\perp\, s_{\lambda+1^\ell}^{(k+1)} = s_\lambda^{(k)}$.

Let $\hat{\sym}^{(k)}$ denote the graded completion
of $\sym/\ZZ \{ m_\lambda \mid \lambda \in \Par \setminus \Par^k\}$.
The space  $\sym_{(k)}$ has basis $\{h_\lambda\}_{\lambda\in \Par^k}$;
$\hat{\sym}^{(k)}$ has ``basis''  $\{m_\lambda\}_{\lambda\in \Par^k}$ meaning that
$\hat{\sym}^{(k)} = \prod_{\lambda \in \Par^k} \Z m_\lambda$.
Let
$\langle \cdot, \cdot \rangle \from \sym_{(k)} \times_\ZZ \hat{\sym}^{(k)} \to \ZZ$
be the bilinear form
determined by
$\langle h_\lambda, \sum_{\mu \in \Par^k} a_\mu m_\mu \rangle = a_\lambda$.
The $K$-$k$-Schur functions $\{g_\lambda^{(k)}\}_{\lambda \in \Par^k} \subset \sym_{(k)}$ and
{affine stable Grothendieck polynomials}
$\{G_\mu^{(k)}\}_{\mu \in \Par^k} \subset \hat{\sym}^{(k)}$ satisfy
$\langle  g_\lambda^{(k)}, G^{(k)}_\mu \rangle = \delta_{\lambda \mu}$.
We take this as the definition of the affine stable Grothendieck polynomials.
For $f\in \hat{\sym}^{(k)}$, let $f^\perp$ be the linear operator on $\sym_{(k)}$ given by
$\langle f^\perp(g),h\rangle = \langle g, fh\rangle$ for all $g \in\sym_{(k)}, h \in \hat{\sym}^{(k)}$.

\begin{thm}[Shift Invariance]
\label{shift-invariance}
For  \(\lambda \in \Par^k_\ell\),
\[
    G_{1^\ell}^\perp\, \g_{\lambda+1^\ell}^{(k+1)} = \g_\lambda^{(k)}
\quad\text{where}\quad
G_{1^\ell} = \sum_{i \geq 0} (-1)^i \binom{\ell-1+i}{\ell-1} e_{\ell+i}\,.
\]
Hence by Theorem \ref{formulations-are-equal},
\(G_{1^\ell}^\perp\, g_{\lambda+1^\ell}^{(k+1)} = g_\lambda^{(k)} \) as well.
\end{thm}
\begin{proof}
We use that \( e_s^\perp h_m = h_m e_s^\perp + h_{m-1} e_{s-1}^\perp \)
  from \cite{garsia-procesi}*{Equation 5.37} to deduce
  \begin{align*}
    e_s^\perp k_m^{(r)}
     = \sum_{i=0}^m \binom{r+i-1}{i}(h_{m-i} e_s^\perp + h_{m-i-1}
      e_{s-1}^\perp)
     = k_m^{(r)} e_s^\perp + k_{m-1}^{(r)} e_{s-1}^\perp
\,.
  \end{align*}
Using that \(e_i^\perp(1) = 0\) for \(i > 0\),
this applies to the formulation for Katalan functions in Proposition~\ref{def:Kat},
giving that, for \(s \geq 0\), \(\Psi \subset \Delta^+\) a root ideal, \(M\)
a multiset with $\supp(M) \subset \{1,\ldots,\ell\}$, and \(\gamma \in \Z^\ell\),
\[
 e_s^\perp K(\Psi;M;\gamma) = \sum_{S \subset [\ell], \
 |S|=s} K(\Psi;M;\gamma-\epsilon_{S})\,,
\]
where $\epsilon_{S} = \sum_{i \in S} \epsilon_i$.
In particular,  $e_\ell^\perp K(\Psi;M;\gamma+1^\ell) = K(\Psi;M;\gamma)$.
Now for  $\lambda \in \Par^k_\ell$, noting that $\Delta^{m}(\lambda+1^\ell)=\Delta^{m-1}(\lambda)$ for
any \(m \geq k+1\), we obtain
$$
    e_\ell^\perp \g_{\lambda+1^\ell}^{(k+1)} = e_\ell^\perp K(\Delta^{k+1}(\lambda+1^\ell);
    \Delta^{k+2}(\lambda+1^\ell);\lambda+1^\ell)
    = K(\Delta^{k}(\lambda);
      \Delta^{k+1}(\lambda);\lambda) = \g_\lambda^{(k)} \,.
$$
Therefore, $e_\ell^\perp \g_{\lambda+1^\ell}^{(k+1)} = \g_\lambda^{(k)}$.
Since \(e_s^\perp K(\Psi;\lowers;\lambda) = 0\) for \(s > \ell\), we can replace
$e_\ell^\perp$ by $G_{1^\ell}^\perp$.
\end{proof}

Shift invariance implies that $K$-$k$-Schur branching coefficients are a subset of the Pieri coefficients for affine stable Grothendieck polynomials, settling Conjecture~\ref{branchingconj}.

\begin{thm}\label{branching-positivity}
\label{thm:gbranching}
For any $\lambda\in\Par^k$,
\begin{equation}
\label{gkingkp1}
    g_\lambda^{(k)} = \sum_{\mu\in\Par^{k+1}} a_{\lambda \mu}\, g_\mu^{(k+1)}
  \quad \text{ where \( (-1)^{|\lambda|-|\mu|}a_{\lambda \mu} \in \Z_{\geq 0}\).}
\end{equation}
\end{thm}

\begin{proof}
  Fix \(\ell = \ell(\lambda)\). For $\mu\in\Par^{k+1}$, Baldwin and
  Kumar~\cite{baldwin-kumar} proved that
\begin{align}
\label{e branching}
 G_{1^\ell}^{(k+1)} G_\mu^{(k+1)} = \sum_\gamma c_{\gamma \mu} \,G_\gamma^{(k+1)}\,
\text{ satisfy } (-1)^{|\gamma|-\ell-|\mu|} c_{\gamma \mu} \in \Z_{\geq 0}\,.
\end{align}
Since $\langle g^{(k+1)}_\alpha, G^{(k+1)}_\beta \rangle = \delta_{\alpha \beta}$ for
 $\alpha, \beta \in \Par^{k+1}$, from \eqref{e branching} we obtain
 \begin{align*}
c_{\gamma \mu}
= \big\langle g_\gamma^{(k+1)}, \sum_\beta c_{\beta \mu} \,G_\beta^{(k+1)}  \big\rangle
=  \big\langle g_\gamma^{(k+1)}, G_{1^\ell}^{(k+1)} G_\mu^{(k+1)}  \big\rangle = \big\langle (G_{1^\ell}^{(k+1)})^\perp
   g_\gamma^{(k+1)}, G_\mu^{(k+1)} \big\rangle\,.
\end{align*}
Therefore, for $\gamma= \lambda+1^\ell$,
\[
\sum_\mu c_{\gamma \mu}
  g_\mu^{(k+1)} =  (G_{1^\ell}^{(k+1)})^\perp g_\gamma^{(k+1)} = g_\lambda^{(k)},
\]
where we can apply Theorem~\ref{shift-invariance} (shift-invariance) to the second equality
because \( G_{1^\ell}^{(k+1)} =  G_{1^\ell}\), verified in Section~\ref{sec:GisGk} of the Appendix.
We thus have that \(a_{\lambda \mu} = c_{\lambda+1^\ell,\mu}\), and the result follows
from~\eqref{e branching}.
\end{proof}

Other properties of  $K$-$k$-Schur functions are readily apparent from the Katalan/raising operator description.
For example, the following property was conjectured in~\cite{morse}; while seemingly simple, it was not apparent from previous descriptions and is
the missing ingredient for resolving conjectures in~\cite{lss,morse,IIM}.

\begin{cor}
\label{k-rectangle-bounded-gk}
\label{con:hookg}
For \(\mu \in \Par^k_\ell\) with \(\mu_1 + \ell - 1 \leq k\),
$g_\mu^{(k)}=g_\mu$\,.
\end{cor}
\begin{proof}
Since \(\Delta^{k}(\mu) = \emptyset = \Delta^{k+1}(\mu)\) when \( k-\mu_1 + 1 \geq \ell \),
the result follows from Definition~\ref{def:KatKks}.
\end{proof}

By iterating branching to obtain an expansion for $g_\lambda^{(k)}$ in terms of $g_\mu^{(a)}$
for large enough $a$ so that Corollary~\ref{k-rectangle-bounded-gk} applies to every term,
we establish~\cite{morse}*{Conjecture 46} as well.

\begin{cor}
\label{cor gk to g}
  For \(\lambda \in \Par^k\),
\[
    g^{(k)}_\lambda = \sum_\mu b_{\lambda \mu}\, g_\mu
\quad
\text{ where \(\;(-1)^{|\lambda|-|\mu|} b_{\lambda \mu} \in \Z_{\geq 0}\).}
\]
\end{cor}

\subsection{Katalan functions for quantum Grothendieck polynomials}
\label{ss Katalan for quantum}
We give some background to explain Conjecture \ref{eq:IIMconj}
and then give a conjectural Katalan description of  $\Phi(\mathfrak G_w^Q)$.

The quantum  $K$-theory ring  $\mathcal{QK}(\Fl_{k+1})$ can be identified with a quotient of
 $\C[z_1,\dots, z_{k+1}, Q_1,\dots, Q_{k}]$
 by \cite{KM, act}; see, e.g., \cite[\S1.1--1.2]{IIM} which includes
 an explicit description of the defining ideal.
 A \emph{$k$-rectangle} is a partition of the form $R_i:=(k+1-i)^i$ for  $i \in [k]$.
Define $\sigma_i = \sum_{\mu \subset R_i} g_\mu$ for  $i \in [k]$, and
set  $\sigma_0 = \sigma_{k+1} = g_{R_0} = g_{R_{k+1}} = 1$.
Ikeda, Iwao, and Maeno give the following description of a $K$-theoretic version of the
Peterson isomorphism \cite[Theorem 1.5]{IIM}\,:
\begin{align*}
\Phi \colon \mathcal{QK}(\Fl_{k+1})[Q_1^{-1},\dots, Q_k^{-1}] & \xrightarrow{\cong} \C\otimes_\ZZ\sym_{(k)}[g_{R_1}^{-1}, \dots, g_{R_k}^{-1}, \sigma_1^{-1}, \dots, \sigma_k^{-1} ]  \\[1mm]  
z_i \mapsto \frac{g_{R_i} \sigma_{i-1}}{g_{R_{i-1}} \sigma_i } &, \ \  Q_i \mapsto \frac{g_{R_{i-1}}g_{R_{i+ 1} }}{g_{R_{i}}^2}.
\end{align*}
Lenart and Maeno defined \cite[Definition 3.18]{LenartMaeno}
the \emph{quantum Grothendieck polynomials}
\break $\{\mathfrak G_w^q(x_1,\dots, x_{k+1}, q_1,\dots, q_{k})\}_{w \in S_{k+1}}$
as the image of the ordinary Grothendieck polynomials $\{\mathfrak G_w\}_{w \in S_{k+1}}$
under a quantization map. The $\mathfrak G_w^q$'s specialize to the  $\mathfrak G_w$ at  $q_1= \cdots q_k = 0$.  Following \cite[\S5.4]{IIM}, we work with $\{\mathfrak G_w^Q(z_1,\dots, z_{k+1}, Q_1,\dots, Q_{k})\}_{w \in S_{k+1}} \subset \mathcal{QK}(\Fl_{k+1})$
which differs from the $\mathfrak G_w^q$'s by the change of variables
$z_i = 1-x_i$ for all  $i \in [k+1]$ and \(Q_i=q_i\) for \(i \in [k]\).

The images $\Phi(\mathfrak G_w^Q)$ are described in terms
of a map $\theta \colon S_{k+1} \to \Par^k$.
For  $w = w_1 \cdots w_{k+1} \in S_{k+1}$ in one-line notation,
the \emph{descent set} of  $w$ is  $\text{\rm Des}(w)=\{i :w_i>w_{i+1}\}$, and
its \emph{inversion sequence} $\Inv(w) \in \ZZ_{\ge 0}^{k}$ is given by
$\Inv_i(w) = \big|\{ j>i : w_i > w_j \}\big|.$
Define an injection $\zeta:S_{k+1}\to \Par^k$ by letting column $i$ of $\zeta(w)$ be
\begin{equation}
\label{perm2part}
 \binom{k+1-i}{2}+\Inv_{i}(w_0 w)\,,
\end{equation}
for all $i\in [k]$, where  $w_0$ denotes the longest element of  $S_{k+1}$.
An element of $\Par^k$ is {\it irreducible} if it
has at most $k-i$ parts of size~$i$, or
 equivalently, it contains no $k$-rectangle as a subsequence.
For any  $\mu \in \Par^k$, define the unique irreducible partition $\mu_{\downarrow}$
by deleting from $\mu$ the $k$-rectangles it contains as a subsequence.
Set $\theta(w) = \zeta(w)_\downarrow$.
By \cite[Lemma~7.3]{ksplitcatalans},  the map $\theta$ is the same as
the map  $\lambda$ from \cite[\S6]{LStoda}, \cite[\S7.1]{IIM}.

The  $k$-conjugate involution on  $\Par^k$ introduced in \cite{LMtaboncores}
can be described as follows: for $\mu \in \Par^k$, its  $k$-conjugate is
$\mu^{\omega_k} = \w^{-1} \circ \tau \circ \w(\mu)$,
for $\tau \colon \eS_{k+1} \to \eS_{k+1}$ the automorphism  given by $s_i \mapsto s_{k+1-i}$.
Note that for  $\mu$ contained in a  $k$-rectangle,  $\mu^{\omega_k}$ is equal to the (ordinary) conjugate
partition  $\mu'$ of  $\mu$.

Ikeda conjectured that the image $\Phi(\mathfrak G_w^Q)$
is in fact not best described with $K$-$k$-Schur functions,
but instead proposed~\cite{ikedaprivate} the functions
$\tilde g_w = (1-G_1^\perp) \Big(\sum_{\mu \in \Par^k,\,
w_\mu\leq w_\lambda } g_\mu^{(k)}\Big)$
from \eqref{gwiggle}.
We conjecture the following explicit raising operator formula for Ikeda's functions.

\begin{defn}\label{def:thickKatKks}
For \(\lambda \in \Par_\ell^k\), the \de{closed $k$-Schur Katalan function} is
\[ \tg_\lambda^{(k)} = K(\Delta^{k}(\lambda); \Delta^{k}(\lambda); \lambda) \,.
  \]
\end{defn}

\begin{conjecture}
\label{c tg}
Let $w\in S_{k+1}$ and $\mu\in\Par^k_\ell$ be arbitrary and set $\lambda=\theta(w)^{\omega_k}$.  Then
  \begin{enumerate}
  \item \(\tg_\lambda^{(k)} = \tilde{g}_w\),
  \item \[\Phi(\mathfrak G_w^Q)=
\frac{\tg_{\lambda}^{(k)}}
{\prod_{d\in \text{\rm Des}(w)} g_{R_d}} \,,\]
\item (alternating dual Pieri rule) the coefficients in $G_{1^m}^\perp \tg_{\mu}^{(k)} = \sum_{\nu} c_{\mu\nu} \tg_\nu^{(k)}$
    satisfy \((-1)^{|\mu|-|\nu|} c_{\mu\nu} \in \Z_{\geq 0}\),
 \item (\(k\)-branching)
 \label{tg-k-branching}
 the coefficients in \(\tg_\mu^{(k)} =
     \sum_\nu a_{\mu\nu} \tg_\nu^{(k+1)}\) satisfy
    \((-1)^{|\mu|-|\nu|} a_{\mu\nu} \in \Z_{\geq 0}\),
  \item (\(K\)-\(k\)-Schur alternating)
  \label{tg-K-k-schur-alt}
the coefficients in  $\tg_\mu^{(k)} = \sum_{\nu} b_{\mu\nu}\, g_\nu^{(k)}$ satisfy
    \((-1)^{|\mu|-|\nu|} b_{\mu\nu} \in \Z_{\geq 0}\),
\item (\(k\)-rectangle property) for \(d \in[k]\),
      \  $g_{R_d} \, \tg_{\mu}^{(k)} = \tg_{\mu \union R_d}^{(k)},$
      where $\mu\union R_d$ is the partition made by
      combining the parts of $\mu$ and those of  $R_d$ and then sorting.
 \end{enumerate}
\end{conjecture}

\begin{rmk}
Conjectures (b) and (e) are just slight variants of previous conjectures in
that, assuming (a), (b) is equivalent to Conjecture \ref{eq:IIMconj} and
(e) is equivalent to the \(K\)-\(k\)-Schur alternating for  $\tilde{g}_w$'s conjectured in \cite{IIM}.  Similarly, Takigiku~\cite{takigiku,takigiku-later} proved a \(k\)-rectangle property for a related family which is equivalent to (f) assuming (a).

Note that   (c) implies (d) by shift invariance (Proposition~\ref{p thick results} (c) below).
\end{rmk}

\begin{example}
Let us directly verify Conjecture~\ref{c tg} (b) for $k=2$ and  $w=213$ (one-line notation),
using the definition of $\Phi$.
The quantum  Grothendieck is $\G_{w}^Q=1-z_1+z_1Q_1$.
Thus using $g_{R_1}=h_2$, $g_{R_2} = h_1^2-h_2+h_1$, $\sigma_1=h_2+h_1+1$, and $\sigma_2=h_1^2-h_2+2h_1+1$,
$$
\Phi(\G_{w}^Q) =
1- \frac{g_{R_1}}{\sigma_1} + \frac{g_{R_1}}{\sigma_1} \frac{g_{R_2}}{g_{R_1}^2}
= \frac{(h_2+h_1+1) h_2-h_2^2+h_1^2-h_2+h_1}{h_2(h_2+h_1+1)}
= \frac{h_1}{h_2}
=\frac{\tg_{(1)}^{(2)}}{g_{R_1}}.
$$
This is the desired conclusion as
$\zeta(w)'=(2,1)$, $\theta(w)=(1)=\theta(w)^{\omega_2}$, and $\text{\rm Des}(w)=\{1\}$.

For  $k=4$ and $v=13254\in S_5$, we use $\Inv(w_0v)=(4,2,2,0)$ to find $\zeta(v)'=(10,5,3,0)$ and $\theta(v)=(3,2,2,1)$.
We have $\theta(v)^{\omega_4}=(3,2,1,1,1)$ and $\text{\rm Des}(v)=\{2,4\}$, so Conjecture~\ref{c tg} (b) states that
$$
\Phi(\G_v^Q) = \frac{\tg_{(3,2,1,1,1)}^{(4)}}{g_{R_2} \, g_{R_4}}\,, %
$$
as can be confirmed in Sage.
\end{example}

\begin{prop}
\label{p thick results}
The closed $k$-Schur Katalan functions $\{\tg_\lambda^{(k)}\}_{\lambda\in\Par^k}$
\begin{enumerate}
\item form a basis for $\Lambda_{(k)}$;
\item are unitriangularly related to $K$-$k$-Schur functions
\begin{equation}
\label{etguni}
\tg_\lambda^{(k)} = g_\lambda^{(k)} + \sum_{\nu:|\nu|<|\lambda|} b_{\lambda\nu}\, g_\nu^{(k)}\,;
\end{equation}
\item satisfy shift invariance
$$ G_{1^\ell}^\perp \tg_{\lambda+1^\ell}^{(k+1)}=\tg_{\lambda}^{(k)}\,;$$
\item simplify as $\tg_\lambda^{(k)}=g_\lambda$ for  $\lambda$ contained in a  $k$-rectangle, i.e., \(\lambda \in \Par^k_\ell\) with \(\lambda_1 + \ell - 1 \leq k\).
    \end{enumerate}
\end{prop}
\begin{proof}
Property (c) is proved just as in Theorem~\ref{shift-invariance}, and (d) just as in Corollary \ref{con:hookg}.
For (a)--(b), a similar result will be proved for the  $\g_\lambda^{(k)}$'s in Proposition~\ref{prop:triangular},
which easily adapts to this setting.
\end{proof}

It is worth pointing out that having the Katalan formulations for (closed) $K$-$k$-Schur functions
readily enables us to complete the proof of Conjecture~\ref{eq:IIMconj}
for Grassmannian permutations outlined in \cite[Theorem~1.7]{IIM}.

\begin{prop}
\label{k-rectangle-bounded-qkt}
Conjecture~\ref{eq:IIMconj} holds for $w\in S_{k+1}$ with  $\text{\rm Des}(w) = \{d\}$.
In fact, in this case, we have
\begin{align}
\label{ep Grass g}
\Phi(\G_w^Q)=\frac{g_{\theta(w)'}}{g_{R_d}} =
\frac{\tilde{g}_{w}}{g_{R_d}} =
\frac{\tg^{(k)}_{\theta(w)'}}{g_{R_d}}=
\frac{g^{(k)}_{\theta(w)'}}{g_{R_d}}\,.
\end{align}
\end{prop}
\begin{proof}
The first equality of \eqref{ep Grass g} is established in Theorem~1.7 and Lemma 7.1 of \cite{IIM}.
The partition $\lambda=\theta(w)^{\omega_k} = \theta(w)'$ lies in a  $k$-rectangle
by \cite[Lemma~7.5]{ksplitcatalans}.
Thus, by Corollary~\ref{con:hookg} and
Proposition~\ref{p thick results} (d),
$g_{\theta(w)'}=g_{\theta(w)'}^{(k)}= \tg_{\theta(w)'}^{(k)}$.
It remains to prove  $g_{\theta(w)'} = \tilde{g}_{w}$.
Using again Corollary~\ref{con:hookg}  on the definition of
 $\tilde{g}_{w}$ in \eqref{gwiggle} gives
\begin{equation*}
\tilde{g}_{w} = (1-G_1^\perp) \big( \textstyle \sum_{w_\mu \le w_\lambda} g_\mu \big)=
(1-G_1^\perp) \big( \textstyle \sum_{\mu \subset \lambda} g_\mu \big) =
g_\lambda,
\end{equation*}
where the second equality follows
using~\cite{LMtaboncores}*{Proposition 40} in addition to the fact that \(\mu\) is equal to the
\((k+1)\)-core of \(\mu\) for \(\mu\)
lying in a \(k\)-rectangle (cores are discussed in
\S\ref{sec:ri-to-core-dict}),
and the last equality holds by the following result of Takigiku~\cite{takigiku-later}:
the map $1-G_1^\perp \colon \Lambda \to \Lambda$ is a ring automorphism
with inverse $F\colon h_i \mapsto \sum_{j \leq i} h_j$ and satisfies  $F(g_\nu) = \sum_{\mu \subset \nu} g_\nu$ for all  $\nu$.
\end{proof}

Another conjecture of~\cite{IIM} about the image of the quantum Grothendieck polynomials is that
\begin{align}\label{econj IIM}
\Phi(\G^Q_{w_0})=\frac{\prod_{i=1}^{k-1} g_{(k-i)^i}}{g_{R_1} \cdots g_{R_k}}\,.
\end{align}
We prove the corresponding result for the closed $k$-Schur Katalan functions:

\begin{prop}
\label{prop:tg-longest-word}
For $w_0$ the longest permutation in  $S_{k+1}$ and $\lambda=\theta(w_0)^{\omega_k}$,
$\tg^{(k)}_{\lambda} = \prod_{i=1}^{k-1} g_{(k-i)^i}.$
\end{prop}

Thus~\eqref{econj IIM} would now follow from Conjecture~\ref{c tg}(b).
Proposition~\ref{prop:tg-longest-word} is proved in \S\ref{proof-of-tg-longest-word}.

\subsection{Positivity conjectures for Katalan functions}

Given a root ideal $\Psi\subset\Delta_\ell^+$ and weight $\gamma\in\mathbb Z^\ell$,  define
\begin{align}
\label{ed maxband}
{\rm maxband}(\Psi,\gamma)=
\max\{\gamma_i+ {\rm nr}(\Psi)_i : i\in[\ell]\}, \quad \text{ for } \, {\rm nr}(\Psi)_i :=  \big|\big\{j \in \{i+1, \dots, \ell\} : (i,j)  \notin \Psi\big\}\big|.
\end{align}
We say $\alpha\in \Psi$ is a \emph{removable root of  $\Psi$}
when $\Psi \setminus \alpha$ is a root ideal and a root $\beta \in \Delta^+
\setminus \Psi$ is \emph{addable to $\Psi$} if $\Psi \cup \beta$ is a root ideal.
Define \(RC(\Psi)\) to be \(\Psi \setminus \{\text{removable roots of }\Psi\}\).
For a nonnegative integer $a$, iteratively define $RC^a(\Psi) = RC(RC^{a-1}(\Psi))$,
starting from $RC^0(\Psi)=\Psi$.

\begin{conjecture}\label{conj:Kpos}
For a root ideal \(\Psi \subset \Delta^+_\ell\) and \(\lambda \in \Par^k_\ell\)
such that maxband$(\Psi,\lambda)\leq k$,
\begin{align}
\label{ec Kpos1}
    K(\Psi;\Psi;\lambda) =
\sum_{\mu\in\Par^k_\ell \atop|\mu|\leq |\lambda|}
\, a_{\lambda \mu}\, \tg^{(k)}_\mu
    \quad\text{for}\quad (-1)^{|\lambda|-|\mu|} a_{\lambda \mu} \in
    \Z_{\geq 0}\,.
\end{align}
For $a\in \Z_{\geq 0}$,
\begin{align}
\label{ec Kpos2}
    K(\Psi;RC^a(\Psi);\lambda) =
\sum_{\mu\in\Par^k_\ell \atop|\mu|\leq |\lambda|}
\, b_{\lambda \mu}\, s^{(k)}_\mu
\quad\text{for}\quad b_{\lambda \mu} \in \Z_{\geq 0}\,.
\end{align}
\end{conjecture}
\begin{rmk}
The large  $k$ limit ($k \ge |\lambda|$) of Conjecture~\ref{conj:Kpos} is already quite strong:
for \(k \geq |\lambda| \geq |\mu|\), we have \(g_\mu^{(k)}
 = g_\mu\)~\cite{morse} and \(s_\mu^{(k)} = s_\mu\)~\cite{LMktab},
  so \eqref{ec Kpos1} and \eqref{ec Kpos2} become conjectures on
  $g_\mu$-alternating and Schur positivity, respectively.
Conjecture \eqref{ec Kpos1}
can be seen as a generalization of
  branching Conjecture~\ref{c tg}(\ref{tg-K-k-schur-alt}) as
  setting $\Psi = \Delta^{k-1}(\lambda)$ gives $K(\Psi;\Psi;\lambda) = \tg^{(k-1)}_\lambda$.
  And Conjecture \eqref{ec Kpos2} can be seen as a vast generalization of the
  conjectured $k$-Schur positivity of the  $g_\lambda^{(k)}$'s posed in
  \cite{lss}*{Conjecture 7.20(1)}.
\end{rmk}

\section{Basic properties of Katalan functions}
 \label{section 3}
We use the notation \([a,b]\) for \(\{i \in \Z \st a \leq i \leq b\}\) and \([n] = [1,n]\).
A \emph{multiset  $M$ on  $[\ell]$} is a multiset whose support is contained in  $[\ell]$;
its multiplicity
function is denoted \(m_M \from [\ell] \to \Z_{\geq 0}\).
For a set \(S \subset [\ell]\), denote \(\epsilon_{S} =
\sum_{i \in S} \epsilon_i\), and for \(\alpha = (i,j) \in
\Delta^+_\ell\), denote by \(\eroot{\alpha} = \epsilon_i - \epsilon_j\)
the corresponding positive root (not to be confused
with \(\epsilon_{\{i,j\}} = \epsilon_i+\epsilon_j\)).

Given a root ideal \(\Psi \subset \Delta^+_\ell\), a multiset \(M\)
on \([\ell]\), and \(\gamma \in \Z^\ell\), we represent the
Katalan function \(K(\Psi;M;\gamma)\) by the \(\ell \times \ell\)
grid of boxes (labelled by matrix-style coordinates)
with the boxes of \(\Psi\) shaded,
\(m_M(a)\) \(\bullet\)'s in column \(a\) (assuming \(m_M(a) < a\)), and
the entries of \(\gamma\) written along the diagonal.
\begin{example}\label{ex:katalan-conventions}
Let \(\Psi = \{(1,3),(1,4),(1,5),(2,3),(2,4),(2,5),(3,5)\}\subset \Delta^+_5\),
\(M = \{2,3,4,4,5,5\}\), and \(\gamma = (3,4,4,2,1)\).
The root ideal $\Psi$, its complement \(\Delta^+ \setminus \Psi = \{(1,2), (3,4), (4,5)\}\),
and $K(\Psi,M,\gamma)$ are depicted by:
 \[
\Psi=
    \ytableausetup{boxsize=1.3em,centertableaux}
    \begin{ytableau}
      {} & {} & *(red) \text{\small 1,3}& *(red) \text{\small 1,4}&
      *(red) \text{\small 1,5} \\
      {} & {} & *(red) \text{\small 2,3}& *(red) \text{\small 2,4}& *(red) \text{\small 2,5}\\
      {} & {} & {} & {} & *(red) \text{\small 3,5} \\
      {} & {} & {} & {} & {} \\
      {} & {} & {} & {} & {}
    \end{ytableau}
\qquad \qquad
\Delta^+_5 \setminus \Psi =
    \begin{ytableau}
      {} & *(lightpurple) \text{\small 1,2} & *(lightred) & *(lightred) & *(lightred) \\
      {} & {} & *(lightred) & *(lightred) & *(lightred) \\
      {} & {} & {} & *(lightpurple) \text{\small 3,4}  & *(lightred) \\
      {} & {} & {} & {} & *(lightpurple) \text{\small 4,5} \\
      {} & {} & {} & {} & {}
    \end{ytableau}
\qquad \qquad
  \ytableausetup{mathmode, boxsize=1.3em,centertableaux,nobaseline}
  K(\Psi;M;\gamma) =
  \begin{ytableau}
  3 & \bullet & *(red) \bullet & *(red) \bullet & *(red) \bullet \\
  {} & 4 & *(red) & *(red) \bullet & *(red) \bullet\\
  {} & {} & 4 & {} & *(red) \\
  {} & {} & {} & 2 & {} \\
  {} & {} & {} & {} & 1
\end{ytableau}
\,.
\]
\end{example}

The raising and lowering operators used in Section \ref{section 2} are informal and
are not well-defined operators on  $\Lambda$ despite their name.  The formal interpretation
of Definition~\ref{def:Kat-g} is as follows:
set $\mathbb{A} = \ZZ \llbracket \frac{z_1}{z_2}, \ldots,\frac{z_{\ell-1}}{z_\ell}\rrbracket [z_1^{\pm 1}, \ldots, z_\ell^{\pm 1}]$,
an arbitrary element of which has the form $\sum_{\gamma \in \ZZ^\ell} c_\gamma \mathbf{z}^\gamma$
where the support $\{\gamma \in \ZZ^\ell \mid c_\gamma \ne 0\}$ is contained in $Q^+ + F$ for some finite subset  $F \subset \ZZ^\ell$,
where $Q^+ :=$  $\ZZ_{\ge 0}\{\epsilon_1-\epsilon_2, \dots,
\epsilon_{\ell-1}-\epsilon_\ell \} \subset \ZZ^\ell$.
For a root ideal \(\Psi \subset \Delta^+_\ell\), multiset \(M\) on \([\ell]\), and
\(\gamma \in \Z^\ell\),
\begin{equation}
  \label{eq:formal-def-of-Kat}
  K(\Psi;M;\gamma) = g\left( \prod_{(i,j) \in \Psi}
    \left(1-\frac{z_i}{z_j}\right)^{-1} \prod_{j \in M}
    \left(1-\frac{1}{z_j}\right) \mathbf{z}^\gamma  \right) \,,
\end{equation}
where \(g \from \mathbb{A} \to \ZZ[h_1,h_2, \ldots]\)
is defined by
\(\sum_{\gamma \in \ZZ^\ell} c_\gamma \mathbf{z}^\gamma \mapsto
\sum_{\gamma \in \ZZ^\ell} c_\gamma g_\gamma\);  note that by \eqref{eq:dual-grothendieck-jacobi-trudi},  $g_\gamma = 0$ when $\gamma_i < i-\ell$, and hence $\sum_{\gamma} c_\gamma g_\gamma$ has finitely many nonzero terms and so indeed lies in  $\ZZ[h_1,h_2, \ldots]$.

Further, defining \(\kappa \from \mathbb{A} \to \ZZ[h_1, h_2, \ldots]\) by
\(\sum_{\gamma \in \ZZ^\ell} c_\gamma \mathbf{z}^\gamma \mapsto
\sum_{\gamma \in \ZZ^\ell} c_\gamma k_\gamma\),
it follows from \eqref{eq:dual-grothendieck-jacobi-trudi} that
\begin{align}
\label{eq:formal2}
  g (f)= \kappa \bigg( \prod_{1\leq i < j\leq \ell} \left(1-\frac{z_i}{z_j}\bigg)
  \cdot f \right)
\end{align}
for all  $f \in \mathbb{A}$.
Note that Proposition~\ref{def:Kat} now follows from~\eqref{eq:formal-def-of-Kat}--\eqref{eq:formal2}.

The symmetric group \(S_\ell\) acts on the ring \(\mathbb{A}\) by
permuting the  $z_i$.  In particular, the simple reflections \(s_1,\dots, s_{\ell-1}\)
act by  \(s_i \big( \sum_\gamma c_\gamma \mathbf{z}^\gamma\big) =
\sum_\gamma c_\gamma \mathbf{z}^{s_i\gamma}\), where \(s_i \gamma = (\gamma_1, \ldots, \gamma_{i-1},
\gamma_{i+1},\gamma_i, \gamma_{i+2}, \ldots)\).
We also consider an action of \(S_\ell\) on subsets \(\Psi \subset [\ell] \times [\ell]\)
defined by \(s_i \Psi = \{(s_i(a), s_i(b)) \st (a,b) \in
\Psi\}\), and an action on multisets \(M\) on  \([\ell]\)
with \(s_i M\) defined by its multiplicity function \(m_{s_i M}(a) = m_M(s_i(a))\) for all \(a \in [\ell]\).

\begin{prop}\label{g-straightening}
For any \(\gamma \in \Z^\ell\),
  \( g_\gamma - g_{\gamma-\epsilon_{i+1}} =
    g_{s_i \gamma - \epsilon_i} - g_{s_i \gamma + \epsilon_{i+1}-\epsilon_{i}}\).
Hence the operator identity
\[
g \circ \bigg( 1-\frac{1}{z_{i+1}} \bigg) \bigg(1+\frac{z_{i+1}}{z_i} s_i \bigg) = 0.
 \]
  \end{prop}
\begin{proof}
Using the definition $g_\gamma = \det(k_{\gamma_i+j-i}^{(i-1)})_{1 \leq i,j \leq \ell}$,
we can write \(g_\gamma - g_{\gamma-\epsilon_{i+1}} = \det(A)\),
for  $A$ the matrix
whose  $i+1$-st row is
$(k_{\gamma_{i+1}+j-i-1}^{(i)} - k_{\gamma_{i+1}+j-i-2}^{(i)})_{j \in [\ell]}$ and whose
other rows agree with the matrix defining  $g_\gamma$;
similarly, we can write \(g_{s_i \gamma + \epsilon_{i+1}-\epsilon_{i}}- g_{s_i \gamma - \epsilon_i} = \det(A') \).
Simplifying the  $i+1$-st rows of  $A$ and  $A'$ using~\eqref{pascal-kh-identity}, we see that
$A$ and  $A'$ differ by swapping their  $i$ and  $i+1$-st rows.
The result follows.
 \end{proof}

\begin{lem}\label{katalan-straightening}
  Let \(\Psi \subset \Delta^+\) be any root ideal and \(M\) on
  \([\ell]\) be any multiset such that
  \begin{enumerate}
  \item \(s_i \Psi = \Psi\) and
  \item \(m_M(i+1) = m_M(i)+1\).
  \end{enumerate}
  Then, for any \(\gamma \in \Z^\ell\), \[
    K(\Psi; M; \gamma) + K(\Psi; M; s_i\gamma - \epsilon_{i}
    + \epsilon_{i+1}) = 0 \,.
  \]
\end{lem}

\begin{proof}
The map \(g\) from~\eqref{eq:formal-def-of-Kat} allows us to
express $K(\Psi; M; \gamma) + K(\Psi; M; s_i\gamma - \epsilon_{i} + \epsilon_{i+1})$ as
$$
g\circ
\left(
1-\frac{1}{z_{i+1}}
\right)
\prod_{(a,b) \in \Psi}
      \left( 1 - \frac{z_a}{z_b} \right)^{-1} \prod_{b \in
      M \setminus \{i+1\}} \left( 1-\frac{1}{z_{b}} \right)
      \left(1+\frac{z_{i+1}}{z_i}
      s_i\right) (\mathbf{z}^\gamma)\,.
$$
Since \(s_i \Psi = \Psi\) and $s_i(M\setminus\{i+1\}) = M\setminus\{i+1\}$,
the operator  $s_i$ commutes with multiplication by
 $\prod_{(a,b) \in \Psi}
      ( 1 - \frac{z_a}{z_b} )^{-1} \prod_{b \in M \setminus \{i+1\}} ( 1-\frac{1}{z_{b}} )$,
      hence so does the operator  $1+\frac{z_{i+1}}{z_i}s_i$.
Therefore $K(\Psi; M; \gamma) + K(\Psi; M; s_i\gamma - \epsilon_{i} + \epsilon_{i+1})$ equals
\begin{equation}
\label{eq:sumK}
     g \circ \left( 1-\frac{1}{z_{i+1}} \right) \left(1+\frac{z_{i+1}}{z_i} s_i\right)
\prod_{(a,b) \in \Psi}
      \left(1-\frac{z_a}{z_b}\right)^{-1} \prod_{b \in M
      \setminus \{i+1\}}
      \left(1-\frac{1}{z_b}\right)\,( \mathbf{z}^\gamma ) \,,
\end{equation}
which vanishes by Proposition~\ref{g-straightening}.
\end{proof}

\begin{lem}
\label{zero-lemma}
  Given a root ideal \(\Psi \subset \Delta^+_{\ell+1}\), a multiset
  \(M\) on \([\ell+1]\), and \(\gamma \in
  \Z^\ell\), we have that \[
    K(\Psi;M;(\gamma,0)) = K(\hat{\Psi};\hat{M};\gamma)\,,
  \]
  where \(\hat{\Psi} := \{(i,j) \in \Psi \st 1 \leq i < j \leq \ell\}\) and
  \(\hat{M} := \{j \in M \st 1 \leq j \leq \ell\}\).
\end{lem}
\begin{proof}
Proposition~\ref{def:Kat} implies that
  \begin{align*}
K(\Psi;M;(\gamma,0))
    = &\prod_{j \in \hat{M}} (1-L_j)
    \prod_{(i,j) \in \Delta_{\ell}^+ \setminus \hat{\Psi}} (1-R_{ij})
    \prod_{j=1}^{m_M(\ell+1)} (1-L_{\ell+1})
    \prod_{\substack{(h,\ell+1) \in \Delta_{\ell+1}^+ \setminus \Psi}}
        (1-R_{h,\ell+1}) \,
      k_{(\gamma,0)}
\\
    =& \prod_{j \in \hat{M}} (1-L_j)
    \prod_{(i,j) \in \Delta_{\ell}^+ \setminus \hat{\Psi}} (1-R_{ij})
\, k_{\gamma}\,,
  \end{align*}
since \(k_0^{(\ell)} = 1\) and \( k_{m}^{(\ell)} = 0\) for \(m < 0\).
\end{proof}
\begin{rmk}
\label{re:mismatched-lengths}
In light of Lemma~\ref{zero-lemma}, we sometimes abuse
notation by saying that, for
\(\ell' \geq \ell\), root ideal \(\Psi \subset \Delta^+_{\ell'}\), multiset
\(M\) on \([\ell']\), and \(\gamma \in \Z^\ell\),
\[
K(\Psi;M;\gamma) := K(\hat{\Psi};\hat{M};\gamma) \,.
\]
\end{rmk}
\begin{lem}\label{lem:gen-pascal-kh-identity}
  For \(r \geq 0, s \geq 1\), and \(\gamma \in \Z^s\), \[
    \prod_{j=r+1}^{r+s} (1-L_j)^r k_{(0^r,\gamma)} = k_\gamma \,.
  \]
\end{lem}
\begin{proof}
  We note that \[
    k_a^{(b-r)} = k_a^{(b-r+1)}-k_{a-1}^{(b-r+1)} = \cdots = \sum_{i=0}^r (-1)^i \binom{r}{i} k_{a-i}^{(b)}
  \]
  by iterating~\eqref{pascal-kh-identity}. Then, \[
    k_\gamma = k_{0^r}k_{\gamma_1}^{(0)}\cdots k_{\gamma_s}^{(s-1)} = k_{0^r} \left( \sum_{i_1=0}^r
      (-1)^{i_1} \binom{r}{i_1} k_{\gamma_1-i_1}^{(r)} \right)\cdots\left( \sum_{i_s=0}^r
      (-1)^{i_s} \binom{r}{i_s} k_{\gamma_s-i_s}^{(r+s-1)} \right)
    = \prod_{j=r+1}^{r+s} (1-L_j)^r k_{(0^r,\gamma)}\,.
  \]
\end{proof}
\begin{defn}
  Given root ideals \(\Psi \subset \Delta^+_\ell\) and \(\Psi' \subset
  \Delta^+_{\ell'}\), we define the root ideal \(\Psi \rootconcat \Psi' \subset
  \Delta^+_{\ell+\ell'}\) to be the result of placing \(\Psi\) and
  \(\Psi'\) catty-corner and including the full \(\ell \times \ell'\)
  rectangle of roots in between. Equivalently, \(\Psi \rootconcat
  \Psi'\) is determined by \[
    \Delta^+_{\ell+\ell'} \setminus (\Psi \rootconcat \Psi') =
    (\Delta^+_\ell \setminus \Psi)\,\, \disjunion\,\, \{(i+\ell,j+\ell) \st (i,j)
    \in \Delta^+_{\ell'} \setminus \Psi\}\,.
  \]
\end{defn}
For example,
  \[
\ytableausetup{mathmode, boxsize=.55em,centertableaux}
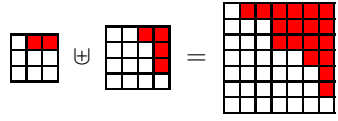
    \begin{ytableau}
      {} & *(red) & *(red) \\
      & {} &  \\
      &   & {}
    \end{ytableau}
\,\,
    \rootconcat
\,\,
    \begin{ytableau}
      {} &  & *(red)& *(red)\\
      & {} & & *(red)\\
      & &  {} &*(red) \\
      & & & {}
    \end{ytableau}
\,\,
    =
\,\,
    \begin{ytableau}
      {} & *(red) & *(red) & *(red) & *(red) & *(red) & *(red) \\
      & {} &  & *(red) & *(red) & *(red) & *(red) \\
      & & {} & *(red) &*(red) & *(red) & *(red) \\
      & & & {} & & *(red) & *(red)  \\
      & & & & {} & & *(red)\\
      & & & & & {} &*(red) \\
      & & & & & & {}
    \end{ytableau}
\]

\begin{lem}\label{katalan-concatenation}
Given $\lambda\in\mathbb Z^\ell,\mu\in\mathbb Z^{\ell'}$, root ideals
\(\Psi, \lowers \subset \Delta^+_\ell\), and root ideals \(\Psi', \lowers' \subset \Delta^+_{\ell'}\),
we have
\[
    K(\Psi; \lowers; \lambda)K(\Psi';\lowers';\mu) =
K(\Psi \rootconcat \Psi'; \lowers \rootconcat \lowers'; \lambda\mu)\,,
  \]
 where \(\lambda\mu=(\lambda_1,\ldots,\lambda_\ell,\mu_1,\ldots,\mu_{\ell'})\).
\end{lem}
\begin{proof}
By Proposition \ref{def:Kat},
  \[
    K(\Psi \rootconcat \Psi'; \lowers \rootconcat \lowers'; \lambda\mu) =
\prod_{(i,j) \in \lowers \rootconcat \lowers'}
    (1-L_j) \prod_{(i,j) \in \Delta^+_{\ell+\ell'} \setminus \Psi
      \rootconcat \Psi'} (1-R_{ij}) k_{\lambda\mu} \,.
  \]
However, since \(\Delta^+_{\ell+\ell'} \setminus \Psi \rootconcat \Psi'\)
has no roots in \(\{(r,s) \st 1 \leq r \leq \ell, \ell+1 \leq s \leq
  \ell+\ell'\}\),
\[
    K(\Psi \rootconcat \Psi'; \lowers \rootconcat \lowers'; \lambda\mu) =
    \prod_{(i,j) \in \lowers \rootconcat \lowers'} (1-L_j)
    \prod_{(i,j) \in \Delta^+_\ell \setminus \Psi} (1-R_{ij})
    \prod_{(i,j) \in \Delta^+_{\ell'} \setminus \Psi'} (1-R_{i+\ell,
      j+\ell}) k_{\lambda \mu}
\,.
  \]
By definition of \(\lowers \rootconcat \lowers'\),
\(\prod_{(i,j) \in \lowers \rootconcat \lowers'} (1-L_j)
= \prod_{(i,j) \in \lowers} (1-L_j)
    \prod_{(i,j) \in \lowers'} (1-L_{\ell+j})
    \prod_{j=\ell+1}^{\ell+\ell'} (1-L_j)^\ell
\).
Noting \(k_{\lambda \mu} = k_\lambda k_{(0^\ell,\mu)}\), we thus have
\begin{align*}
    \hspace{-0.3in}
    K(\Psi \rootconcat \Psi'; \lowers \rootconcat \lowers'; \lambda\mu) =
&
\prod_{(i,j) \in \lowers} (1-L_j)
      \prod_{(i,j) \in \Delta^+_\ell
        \setminus \Psi} (1-R_{ij})\, k_\lambda
\\
&
     \times
     \prod_{j=\ell+1}^{\ell+\ell'} (1-L_j)^\ell
     \prod_{(i,j)
        \in \lowers'} (1-L_{\ell+j})
      \prod_{(i,j) \in \Delta^+_{\ell'}
     \setminus \Psi'} (1-R_{i+\ell, j+\ell})
       k_{(0^\ell, \mu)} \,.
\end{align*}
  The first line is $K(\Psi;\lowers;\mu)$. To see the second line is  $K(\Psi';\lowers';\mu)$, expand
  \(
     \prod_{(i,j)
        \in \lowers'} (1-L_{\ell+j})
      \prod_{(i,j) \in \Delta^+_{\ell'}
     \setminus \Psi'} (1-R_{i+\ell, j+\ell})
       k_{(0^\ell, \mu)}  = \sum_\gamma k_{(0^\ell,\gamma)},
  \)
  and note for each summand, $\prod_{j=\ell+1}^{\ell+\ell'} (1-L_j)^\ell k_{(0^\ell,\gamma)} = k_\gamma$
  by Lemma~\ref{lem:gen-pascal-kh-identity}.
\end{proof}

\begin{prop}
\label{root-expansions}
  Let \(\Psi \subset \Delta^+\) be a root ideal, \(M\) on \([\ell]\) be a multiset, and \(\mu \in
  \Z^\ell\). Then,
  \begin{enumerate}
  \item for any addable root \(\beta\) of \(\Psi\), \[
      K(\Psi;M;\mu) = K(\Psi \union \beta; M; \mu) -
      K(\Psi \union \beta; M; \mu + \eroot{\beta})\,;
    \]
  \item for any removable root \(\alpha\) of \(\Psi\), \[
      K(\Psi;M;\mu) = K(\Psi \setminus \alpha; M; \mu) +
      K(\Psi; M; \mu+\eroot{\alpha})\,;
    \]
  \item for any $y \in M$,
\[ K(\Psi;M;\mu) = K(\Psi;M\setminus y;\mu) -
      K(\Psi;M \setminus y;\mu-\epsilon_{y})\,;
    \]
  \item for any $y \in [\ell]$,
\[
      K(\Psi;M;\mu) = K(\Psi;M\dunion y;\mu) + K(\Psi;M;\mu-\epsilon_{y})\,.
    \]
  \end{enumerate}
\end{prop}
\begin{proof}
The first identity follows directly from Proposition~\ref{def:Kat}:
\[ K(\Psi;M;\mu)
     = \prod_{j \in M}(1-L_j) \prod_{(i,j) \in
      \Delta^+\setminus\Psi}(1-R_{ij}) k_\mu
     = \prod_{j \in M}(1-L_j) \prod_{(i,j) \in
      \Delta^+ \setminus (\Psi \union \beta)}(1-R_{ij})
      (k_\mu-k_{\mu+\eroot{\beta}})\,.
\]
Part (b) is then obtained by applying (a) with \(\Psi = \Psi
\setminus \alpha\) and \(\beta = \alpha\).
A similar computation gives (c):
\[
  K(\Psi;M;\mu)
    = \prod_{j \in M}(1-L_j) \prod_{(i,j) \in
      \Delta^+\setminus\Psi}(1-R_{ij}) k_\mu\\
    = \prod_{j \in M \setminus y}(1-L_j)\prod_{(i,j) \in
      \Delta^+\setminus\Psi}(1-R_{ij}) (k_\mu - k_{\mu-\epsilon_{y}})\,,
\]
and (d) is obtained by applying (c) with $M \dunion \{y\}$ in place of $M$.
\end{proof}

These root expansions give rise to other powerful identites, derived by their successive application.

\begin{lem}\label{sliding-lemma}
Let \(\Psi \subset \Delta_\ell^+\), \(M\) be a multiset on \([\ell]\), and \(\mu \in \Z^\ell\) with \(\mu_\ell = 1\).
If $\ell\in M$ and \(\Psi\) has a removable root \(\alpha = (x,\ell)\) for some $x$, then
\[
    K(\Psi;M;\mu) = K(\Psi \setminus \alpha; M \setminus
    \ell; \mu) + K(\hat{\Psi} ; \hat{M} \dunion x;(\mu_1,\ldots,\mu_{\ell-1})+\epsilon_x)\,,
  \]
where $\hat{\Psi}=\{(i,j)\in\Psi \st j<\ell\}$ and $\hat{M}
=\{j \in M \st j<\ell\}$.
\end{lem}
\begin{proof}
By Proposition~\ref{root-expansions},
we expand first on the removable root \(\alpha=(x,\ell)\) of \(\Psi\) and then on $\ell \in M$, to obtain
\[
    K(\Psi;M;\mu) = K(\Psi \setminus \alpha ;M;\mu) + K(\Psi;M;\mu+ \eroot{\alpha}) =
K(\Psi \setminus \alpha; M \setminus \ell; \mu)
-    K(\Psi \setminus \alpha; M \setminus \ell; \mu-\epsilon_\ell) +
K(\Psi; M; \mu+\eroot{\alpha})\,.
  \]
Lemma~\ref{zero-lemma} allows the substitution of $K(\Psi \setminus \alpha; M \setminus \ell; \mu-\epsilon_\ell)=
    K(\hat{\Psi} ; \hat{M} ; \hat \mu)$ for $\hat \mu=(\mu_1,\ldots,\mu_{\ell-1})$, as well as
    $K(\Psi ; M ; \mu+\eroot{\alpha})= K(\hat{\Psi} ; \hat{M} ; \hat \mu + \epsilon_x)$.
Proposition~\ref{root-expansions}(c) on column $x$ then gives
$- K(\hat{\Psi} ; \hat{M} ; \hat \mu)  + K(\hat{\Psi}; \hat{M}; \hat\mu+\epsilon_x)
= K(\hat{\Psi}; \hat{M}\dunion x; \hat \mu+\epsilon_x)$.
\end{proof}

\begin{example}
  We apply Lemma~\ref{sliding-lemma} to the following scenario, with  $\ell = 7$  and root \(\alpha = (4,7)\):
  \[
    \ytableausetup{mathmode, boxsize=.9em,centertableaux,nobaseline}
    \begin{ytableau}
      4 & *(red) & *(red)\bullet & *(red)\bullet & *(red)\bullet & *(red)\bullet & *(red)\bullet \\
      {} & 3 & {} & *(red) & *(red)\bullet & *(red)\bullet & *(red)\bullet \\
      {} & {} & 1 & {} & {} & *(red)\bullet & *(red)\bullet \\
      {} & {} & {} & 1 & {} & {} & *(red)\bullet \\
      {} & {} & {} & {} & 1 & {} & {} \\
      {} & {} & {} & {} & {} & 1 & {} \\
      {} & {} & {} & {} & {} & {} & 1
    \end{ytableau}
    =
    \begin{ytableau}
      4 & *(red) & *(red)\bullet & *(red)\bullet & *(red)\bullet & *(red)\bullet & *(red)\bullet \\
      {} & 3 & {} & *(red) & *(red)\bullet & *(red)\bullet & *(red)\bullet \\
      {} & {} & 1 & {} & {} & *(red)\bullet & *(red)\bullet \\
      {} & {} & {} & 1 & {} & {} & {} \\
      {} & {} & {} & {} & 1 & {} & {} \\
      {} & {} & {} & {} & {} & 1 & {} \\
      {} & {} & {} & {} & {} & {} & 1
    \end{ytableau}
    +
    \begin{ytableau}
      4 & *(red) & *(red)\bullet & *(red)\bullet & *(red)\bullet & *(red)\bullet\\
      {} & 3 & {} & *(red)\bullet & *(red)\bullet & *(red)\bullet\\
      {} & {} & 1 & {} & {} & *(red)\bullet \\
      {} & {} & {} & 2 & {} & {} \\
      {} & {} & {} & {} & 1 & {} \\
      {} & {} & {} & {} & {} & 1
    \end{ytableau}
\]
\end{example}
\section{Mirror lemmas and straightening relations}

Although a Schur function can be associated to generic $\gamma\in\mathbb Z^\ell$, $s_\gamma$ always either vanishes or
straightens into a single $s_\mu$, up to sign, for a partition $\mu$.
Lemma~\ref{katalan-straightening} shows that Katalan functions satisfy
a straightening relation as well. From this, we deduce adaptations of
the \emph{mirror
  lemmas} of~\cite{catalans} to the \(K\)-theoretic setting and some useful consequences.

\subsection{Root ideal combinatorics}
We begin by reviewing some notation from~\cite{catalans}.

Let $\Psi \subset \Delta^+_\ell$ be a root ideal and $x\in[\ell]$.
If there is a removable root  $(x,j)$ of  $\Psi$, then define $\down_\Psi(x) = j$; otherwise, $\down_\Psi(x)$ is undefined.
Similarly, if there is a removable root $(i,x)$ of  $\Psi$, then define $\upp_\Psi(x) = i$; otherwise, $\upp_\Psi(x)$ is undefined.
The \emph{bounce graph} of a root ideal  $\Psi \subset \Delta^+_\ell$ is the graph on the vertex set $[\ell]$
with edges $(r, \down_\Psi(r))$ for each $r\in [\ell]$ such that $\down_\Psi(r)$ is defined.
The bounce graph of  $\Psi$ is a disjoint union of paths called \emph{bounce paths of  $\Psi$}.

For each vertex $r \in [\ell]$,
distinguish \(\chainup_\Psi(r)\) to be the minimum element of the
bounce path of \(\Psi\) containing \(r\).
For  $a,b \in [\ell]$ in the same bounce path of  $\Psi$ with  $a\le b$, we define
\[ \bpath_\Psi(a,b) = \{a, \down_\Psi(a), \down^2_\Psi(a), \dots, b\}, \]
i.e., the set of indices in this path lying between  $a$ and  $b$.
We also set
$\uppath_\Psi(r)$ to be %
$\bpath_\Psi(\chainup_\Psi(r),r)$ for any $r \in [\ell]$.

\begin{example}
\label{ex:downpath}
A $\bpath$ and $\uppath$ for the root ideal $\Psi$ are given below:
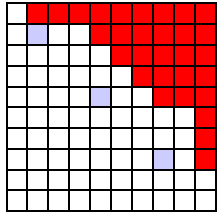
\ytableausetup{mathmode, boxsize=1.03em,centertableaux}
\[\begin{array}{cccccc}
{\tiny \begin{ytableau}
~ & *(red) & *(red)& *(red) & *(red)&*(red)&*(red)&*(red)&*(red)&*(red)\\
~ & *(blue!20) & & & *(red)&*(red)&*(red)&*(red)&*(red)&*(red)\\
~ & & & & &*(red)& *(red) & *(red) &*(red)&*(red)\\
~ & & & & & &*(red)&*(red)&*(red)&*(red)\\
~ & & & & *(blue!20) & & &*(red)&*(red)&*(red)\\
~ & & & & & & & & &*(red)\\
~ & & & & & & & & &*(red) \\
~ & & & & & & & *(blue!20) & & *(red)\\
~ & & & & & & & & & \\
~ & & & & & & & & &
\end{ytableau} } & & &
     & &
{\tiny \begin{ytableau}
*(blue!20) & *(red) & *(red)& *(red) & *(red)&*(red)&*(red)&*(red)&*(red)&*(red)\\
~ & *(blue!20) & & & *(red)&*(red)&*(red)&*(red)&*(red)&*(red)\\
~ & & & & &*(red)& *(red) & *(red) &*(red)&*(red)\\
~ & & & & & &*(red)&*(red)&*(red)&*(red)\\
~ & & & & *(blue!20) & & &*(red)&*(red)&*(red)\\
~ & & & & & & & & &*(red)\\
~ & & & & & & & & &*(red) \\
~ & & & & & & & *(blue!20) & & *(red)\\
~ & & & & & & & & & \\
~ & & & & & & & & & *(blue!20)
\end{ytableau} }\\[16.4mm]
\bpath_{\Psi}(2,8) = \{{\color{blue}2,5,8 } \}& & & %
                            & &  \uppath_{\Psi}(10) =
                                \{{\color{blue}10,8,5,2,1} \}
\end{array}\]
\end{example}

\begin{definition}
For a root ideal $\Psi$, we say there is

\emph{a wall in rows $r,r+1,\dots, r+d$} if rows $r, \dots, r+d$ of $\Psi$ have the same length,

\emph{a ceiling in columns $c,c+1, \dots, c+d$} if columns $c, \dots, c+d$ of $\Psi$ have the same length,

\emph{a mirror in rows} $r,r+1$ if $\Psi$ has removable roots $(r,c)$, $(r+1,c+1)$ for some $c > r+1$.
\end{definition}

\begin{example}
In Example~\ref{ex:downpath}, the root ideal $\Psi$ has
a ceiling in columns  $2,3,4$,  and in columns  $8,9$,
a wall in rows  $6,7,8$, and in rows  $9, 10$, and  a mirror in rows  $2,3$, in rows  $3, 4$, and in rows  $4, 5$.
\end{example}

\subsection{Mirror Lemmas}

\begin{lem}\label{base-case-mirror-lemma}\label{base-case-2-mirror-lemma}
Suppose a root ideal \(\Psi \subset \Delta^+_\ell \), a multiset \(M\) on
\([\ell]\), \(\mu \in \Z^\ell\), and
\(z \in [\ell-1]\) satisfy
  \begin{enumerate}
  \item \(\Psi\) has a ceiling in columns \(z,z+1\);
  \item \(\Psi\) has a wall in rows \(z,z+1\);
  \item \(\mu_z = \mu_{z+1}-1\).
  \end{enumerate}
If $m_M(z+1) = m_M(z)+1$, then \(K(\Psi;M;\mu) = 0\).
If $m_M(z)=m_M(z+1)$, then \(K(\Psi;M;\mu) = K(\Psi;M;\mu-\epsilon_{z+1})\),
\end{lem}
\begin{proof}
  Conditions (a) and (b) are equivalent to \(s_z \Psi = \Psi\) and
  condition (c) implies \mbox{\(\mu = s_z \mu - \epsilon_{z} +
  \epsilon_{z+1}\)}.
Thus, if \(m_M(z+1) = m_M(z)+1\),
the result follows from Lemma~\ref{katalan-straightening}.
If \(m_M(z+1) = m_M(z)\),
Lemma~\ref{root-expansions}(d) implies that
\[
    K(\Psi;M;\mu) = K(\Psi;M\dunion\{z+1\};\mu)+K(\Psi;M;\mu-\epsilon_{z+1})\,.
  \]
By the preceding case, $K(\Psi;M\dunion \{z+1\};\mu)$ vanishes.
\end{proof}
\begin{example}
For $z=2$, Lemma~\ref{base-case-mirror-lemma} applies in the following two situations:
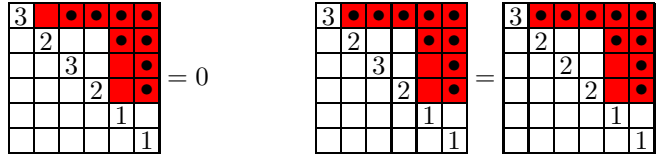
\[
    \ytableausetup{boxsize=0.9em, centertableaux,nobaseline}
    \begin{ytableau}
      3 & *(red) & *(red)\bullet & *(red)\bullet & *(red)\bullet & *(red)\bullet \\
      {} & 2 & {} & {} & *(red)\bullet & *(red)\bullet \\
      {} & {} & 3 & {} & *(red) & *(red)\bullet \\
      {} & {} & {} & 2 & *(red) & *(red)\bullet \\
      {} & {} & {} & {} & 1 & {} \\
      {} & {} & {} & {} & {} & 1
    \end{ytableau}
    = 0
\qquad\qquad
    \begin{ytableau}
      3 & *(red)\bullet & *(red)\bullet & *(red)\bullet & *(red)\bullet & *(red)\bullet \\
      {} & 2 & {} & {} & *(red)\bullet & *(red)\bullet \\
      {} & {} & 3 & {} & *(red) & *(red)\bullet \\
      {} & {} & {} & 2 & *(red) & *(red)\bullet \\
      {} & {} & {} & {} & 1 & {} \\
      {} & {} & {} & {} & {} & 1
    \end{ytableau}
    =
    \begin{ytableau}
      3 & *(red)\bullet & *(red)\bullet & *(red)\bullet & *(red)\bullet & *(red)\bullet \\
      {} & 2 & {} & {} & *(red)\bullet & *(red)\bullet \\
      {} & {} & 2 & {} & *(red) & *(red)\bullet \\
      {} & {} & {} & 2 & *(red) & *(red)\bullet \\
      {} & {} & {} & {} & 1 & {} \\
      {} & {} & {} & {} & {} & 1
    \end{ytableau}
  \]
\end{example}

\begin{lem}[Mirror Lemma]\label{mirror-lemma}
  Let \(\Psi \subset \Delta^+_\ell\) be a root ideal, \(M\) a multiset on
  \([\ell]\), \(\mu \in \Z^\ell\), and \(1 \leq y \leq z < \ell\) be indices in the same bounce path of $\Psi$
satisfying
  \begin{enumerate}
  \item \(\Psi\) has a ceiling in columns \(y,y+1\);
  \item \(\Psi\) has a mirror in rows \(x,x+1\) for all \(x \in
    \path_\Psi(y, \up_\Psi(z))\);
  \item \(\Psi\) has a wall in rows \(z,z+1\);
  \item \(m_M(x+1) = m_M(x)+1\)
    for all \(x \in \path_\Psi(\down_\Psi(y), z)\);
  \item \(\mu_x = \mu_{x+1}\) for all \(x \in \path_\Psi(y,\up_\Psi(z))\);
  \item \(\mu_z = \mu_{z+1}-1\).
  \end{enumerate}
  If \(m_M(y+1) = m_M(y)+1\), then \(K(\Psi;M;\mu) = 0\).  If
  \(m_M(y+1) = m_M(y)\), then \(K(\Psi;M;\mu) = K(\Psi;M;\mu-\epsilon_{z+1})\).
\end{lem}
\begin{proof}
We proceed by induction on \(z-y\), with Lemma~\ref{base-case-mirror-lemma} giving the base case  \(z=y\).
Assume \(z > y\).  Condition (b) implies that \(\up_\Psi(z+1)=\up_\Psi(z)+1\) and
thus the root \(\beta = (\up_\Psi(z+1),z)\) is addable to $\Psi$.  Proposition~\ref{root-expansions}(a) thus
implies that $K(\Psi; M; \mu) = K(\Psi \union \beta; M; \mu) - K(\Psi \union \beta; M; \mu+\eroot{\beta})$.
The root ideal \(\Psi \union \beta\) has a ceiling in columns \(z,z+1\)
and so   \(K(\Psi \union \beta; M; \mu)=0\) by Lemma~\ref{base-case-mirror-lemma}.
Therefore,
\[
 K(\Psi; M; \mu) = - K(\Psi \union \beta; M; \mu+\eroot{\beta})\,.
\]
Because there is a wall in rows \(\up_\Psi(z), \up_\Psi(z+1)\) of the root ideal \(\Psi \union \beta\),
$K(\Psi \union \beta; M; \mu+\eroot{\beta})$ can be addressed by induction:
when \(m_M(y+1) = m_M(y)+1\),
$K(\Psi \union \beta; M; \mu+\eroot{\beta})=0$ implies the vanishing of $K(\Psi; M; \mu)$,
and otherwise, $K(\Psi \union \beta; M; \mu+\eroot{\beta})=
K(\Psi \union \beta; M; \mu+\eroot{\beta}-\epsilon_{\up_\Psi(z)+1}) = K(\Psi \union \beta; M; \mu-\epsilon_z) $
gives
\(
K(\Psi; M; \mu) =
-K(\Psi \union \beta; M; \mu-\epsilon_z)\,.
\)
We then use Lemma~\ref{katalan-straightening} with  $i=z$ to find
\[
K(\Psi; M; \mu) =
-K(\Psi \union \beta; M; \mu-\epsilon_z)=
 K(\Psi \union \beta; M; \mu-\epsilon_{z+1})\,.
\]
Now expand the right hand side
on the removable root \(\beta \in \Psi\union\beta\)
with Proposition~\ref{root-expansions}(b) to obtain
\[
K(\Psi; M; \mu) =
    K(\Psi \union \beta; M; \mu-\epsilon_{z+1}) = K(\Psi;M;\mu-\epsilon_{z+1})+K(\Psi\union\beta;M;\mu-\epsilon_{z+1}+\eroot{\beta})\,.
  \]
Finally, \(K(\Psi\union\beta;M;\mu-\epsilon_{z+1}+\eroot{\beta})\)
vanishes by Lemma~\ref{base-case-mirror-lemma} since
 \(\Psi \union \beta\) has a wall in rows \(z,z+1\)
  and a ceiling in columns \(z,z+1\) and \(\mu-\epsilon_{z+1}+\eroot\beta\) satisfies the
necessary conditions.
\end{proof}

\begin{lem}\label{baby-staircase-lemma}
Suppose a root ideal \(\Psi \subset \Delta_{\ell}^+\), multiset \(M\) on \([\ell]\),
\(\gamma \in \Z^\ell\), and \(j \in [\ell]\) satisfy
  \begin{enumerate}
\item $\Psi$ has a removable root $(i,j)$ in column $j$;
  \item \(\Psi\) has a ceiling in columns \(j,j+1\) and a wall in rows \(j,j+1\);
  \item \(m_M(j+1) = m_M(j)+1\);
  \item \(\gamma_j = \gamma_{j+1}\).
  \end{enumerate}
  Then, \( K(\Psi;M;\gamma) = K(\Psi \setminus (i,j);M;\gamma)\).
\end{lem}
\begin{proof}
A root expansion on the removable root \((i,j)\) with
Proposition~\ref{root-expansions} gives \[
    K(\Psi;M;\gamma) = K(\Psi \setminus (i,j);M; \gamma)
    + K(\Psi;M; \gamma+\epsilon_{i}-\epsilon_{j})\,,
  \]
and the second summand vanishes by Lemma~\ref{base-case-mirror-lemma}
with \(z=j\).
\end{proof}

\begin{lem}\label{staircase-lemma}
  Suppose a root ideal \(\Psi \subset \Delta_{\ell}^+\), multiset \(M\) on \([\ell]\),
  \(\gamma \in \Z^\ell\), and \(j \in [\ell]\) satisfy
  \begin{enumerate}
  \item \(j \in M\);
  \item \(\Psi\) has a ceiling in columns \(j,j+1\) and a wall in rows \(j,j+1\);
\item \(m_M(j+1) = m_M(j)\);
  \item \(\gamma_j = \gamma_{j+1}\).
  \end{enumerate}
  Then, \(K(\Psi;M;\gamma) = K(\Psi;M \setminus j;\gamma)\,.  \)
If, in addition, $\Psi$ has a removable root \((i,j)\) in column  $j$, then
  \(
    K(\Psi;M;\gamma) = K(\Psi \setminus (i,j); M
    \setminus j;\gamma)\,.
  \)
\end{lem}
\begin{proof}
We expand on $j\in M$
with Proposition~\ref{root-expansions} to obtain
\[
    K(\Psi; M; \gamma) = K(\Psi; M \setminus j;
    \gamma) - K(\Psi; M \setminus j;
    \gamma - \epsilon_j)
\,,
  \]
and note that
  \(K(\Psi; M \setminus j; \gamma - \epsilon_{j})=0\)
by Lemma~\ref{base-case-mirror-lemma}.
    If, in addition, \((i,j)\) is removable from \(\Psi\),
the second equality holds since
    \(K(\Psi; M \setminus j;\gamma)\) satisfies the
    conditions of Lemma~\ref{baby-staircase-lemma}.
  \end{proof}
  \begin{example}
  By the first equality of Lemma~\ref{staircase-lemma},  \[
    \ytableausetup{nobaseline}
    \begin{ytableau}
      4 & {} & *(red)\bullet & *(red)\bullet \\
      {} & 3 & *(red)\bullet & *(red)\bullet \\
      {} & {} & 2 & {} \\
      {} & {} & {} & 2
    \end{ytableau}
    =
    \begin{ytableau}
      4 & {} & *(red)\bullet & *(red)\bullet \\
      {} & 3 & *(red) & *(red)\bullet \\
      {} & {} & 2 & {} \\
      {} & {} & {} & 2
    \end{ytableau}
  \]
  By Lemma~\ref{baby-staircase-lemma},  \[
    \ytableausetup{nobaseline}
    \begin{ytableau}
      4 & {} & *(red)\bullet & *(red)\bullet \\
      {} & 3 & *(red) & *(red)\bullet \\
      {} & {} & 2 & {} \\
      {} & {} & {} & 2
    \end{ytableau}
    =
    \begin{ytableau}
      4 & {} & *(red)\bullet & *(red)\bullet \\
      {} & 3 & {} & *(red)\bullet \\
      {} & {} & 2 & {} \\
      {} & {} & {} & 2
    \end{ytableau}
  \]
  Combining both these equalities is an application of the second
  equality of Lemma~\ref{staircase-lemma}.
\end{example}

  \begin{lem}[Mirror Straightening Lemma]\label{kk-schur-straightening}
  Let \(\Psi \subset \Delta^+_\ell\) be a root ideal, \(M\) a multiset on
  \([\ell]\), and \(\mu \in \Z^\ell\).
Let \(1 \leq y \leq z < \ell\) be indices in the same bounce path of $\Psi$ satisfying
  \begin{enumerate}
  \item \(m_M(y) = m_M(y+1)\);
  \item \(\Psi\) has an addable root \(\alpha = (\up_\Psi(y+1),y)\)
    and a removable root \(\beta = (\up_\Psi(y+1),y+1)\);
  \item \(\Psi\) has a mirror in rows \(x,x+1\) for all \(x \in
    \path_\Psi(y, \up_\Psi(z))\); \label{psi-mirror}
  \item \(\Psi\) has a wall in rows \(z,z+1\);
  \item \(m_M(x+1) = m_M(x)+1\)
    for all \(x \in \path_\Psi(\down_\Psi(y), z)\);
  \item \(\mu_x = \mu_{x+1}\) for all \(x \in \path_\Psi(y,\up_\Psi(z))\), and \(\mu_z = \mu_{z+1}-1\).
  \end{enumerate}
  Then, \[
    K(\Psi;M;\mu) = K(\Psi \union \alpha;M \dunion
    (y+1);\mu+\epsilon_{\up_\Psi(y+1)}-\epsilon_{z+1}) +
    K(\Psi;M;\mu-\epsilon_{z+1}) \,.
  \]
\end{lem}
\begin{proof}
First consider the case $z=y$.  We have
$K(\Psi;M;\mu) = K(\Psi;M \dunion (y+1);\mu)+K(\Psi;M;\mu-\epsilon_{z+1})$
by Proposition~\ref{root-expansions}(d), and must prove
that
\begin{equation}
\label{l1}
K(\Psi;M \dunion (y+1);\mu)=K(\Psi \union \alpha; M \dunion (y+1); \mu+\epsilon_{\up_\Psi(z+1)}-\epsilon_{z+1})\,.
\end{equation}
Since \(\alpha = (\up_\Psi(z+1),z) \) is addable to $\Psi$, we expand with Proposition~\ref{root-expansions}
to obtain
\[
K(\Psi;M \dunion (y+1);\mu)
=
K(\Psi \union \alpha; M \dunion (y+1); \mu)
 - K(\Psi \union \alpha; M \dunion (y+1); \mu+\eroot{\alpha})\,.
\]
Conditions (b) and (d) imply that $\Psi\union\alpha$ has a ceiling in columns $y,y+1$
and a wall in rows  $y,y+1$,
and (a) gives that $M\dunion (y+1)$ has one more occurrence of $y+1$ than $y$.
Therefore, since $\mu_z=\mu_{z+1}-1$, Lemma~\ref{katalan-straightening} with $i=y=z$ applies
and straightens the term
\[
    -K(\Psi \union \alpha; M \dunion (y+1); \mu+\eroot{\alpha}) =
    K(\Psi \union \alpha; M \dunion (y+1); \mu+\epsilon_{\up_\Psi(z+1)}-\epsilon_{z+1})\,.
  \]
For the same reasons,
Lemma~\ref{base-case-mirror-lemma} applies to the other term, giving
\(K(\Psi \union \alpha; M \dunion (y+1); \mu)=0\). Thus~\eqref{l1} is proved.

Proceed by induction for $z-y>0$.
Given $\Psi$ has a mirror in rows $w=\up_\Psi(z)$ and $w+1$,
the root \(\gamma = (w+1,z)\) is addable to  $\Psi$ and expanding on it using Proposition~\ref{root-expansions} yields
\[
    K(\Psi;M;\mu) = K(\Psi \union \gamma; M; \mu)-K(\Psi \union \gamma; M; \mu+\eroot{\gamma})\,.
  \]
Since $\Psi \union \gamma$ has a ceiling in columns $z,z+1$, with conditions (d) and (e),
Lemma~\ref{katalan-straightening} straightens the term
\[
    -K(\Psi \union \gamma; M; \mu+\eroot{\gamma}) = K(\Psi \union \gamma; M; \mu+\epsilon_{w+1}-\epsilon_{z+1}) \,.
\]
The same conditions imply that $K(\Psi \union \gamma; M; \mu)=0$ by
Lemma~\ref{base-case-mirror-lemma}.  Therefore,
\[
    K(\Psi;M;\mu) = K(\Psi \union \gamma; M; \mu+\epsilon_{w+1}-\epsilon_{z+1}) \,.
  \]
Since $\Psi\union\gamma$ has a wall in rows $w$ and $w+1$, and
$\nu=\mu+\epsilon_{w+1}-\epsilon_{z+1}$ satisifies $\nu_{w}=\nu_{w+1}-1$,
we can apply the induction hypothesis with \(z=w\) to the right hand side and
obtain
\[
    K(\Psi;M;\mu) =
    K(\Psi \union \{\gamma, \alpha\}; M \dunion (y+1);
    \mu+\epsilon_{\up_\Psi(y+1)} - \epsilon_{z+1}) + K(\Psi \union
    \gamma; M; \mu-\epsilon_{z+1})\,.
  \]
Lemma~\ref{baby-staircase-lemma} enables us to remove \(\gamma\) from both terms, proving the claim.
\end{proof}

\begin{example}
  The following is an example of an application
  of Lemma~\ref{kk-schur-straightening} with \(y=2,z=5\). \[
    \ytableausetup{boxsize=0.9em, centertableaux, nobaseline}
    \begin{ytableau}
  5 & {} &*(red) & *(red)\bullet & *(red)\bullet & *(red)\bullet \\
  {} & 4 & {} & {} & *(red) & *(red)\bullet \\
  {} & {} & 4 & {} & {} & *(red) \\
  {} & {} & {} & 4 & {} & {} \\
  {} & {} & {} & {} & 3 & {}\\
  {} & {} & {} & {} & {} & 4
\end{ytableau}
 = \begin{ytableau}
  6 & *(red) &*(red)\bullet & *(red)\bullet & *(red)\bullet & *(red)\bullet \\
  {} & 4 & {} & {} & *(red) & *(red)\bullet \\
  {} & {} & 4 & {} & {} & *(red) \\
  {} & {} & {} & 4 & {} & {} \\
  {} & {} & {} & {} & 3 & {} \\
  {} & {} & {} & {} & {} & 3
\end{ytableau}
+
\begin{ytableau}
  5 & {} &*(red) & *(red)\bullet & *(red)\bullet & *(red)\bullet \\
  {} & 4 & {} & {} & *(red) & *(red)\bullet\\
  {} & {} & 4 & {} & {} & *(red) \\
  {} & {} & {} & 4 & {} & {} \\
  {} & {} & {} & {} & 3 & {} \\
  {} & {} & {} & {} & {} & 3 \\
\end{ytableau}
  \]
\end{example}

For $1\leq x<y\leq z\le \ell$, define the  {\it diagonal} $D^z_{x,y} = \{(i,j) \st j-i=y-x \,,\, y\leq j\leq z\} \subset \Delta^+_\ell$.
\begin{example}\label{ex:diagonals}
In the following, \(D_{3,4}^6\) is the light blue (removable) diagonal
and \(D_{2,4}^6\) is depicted in dark blue.
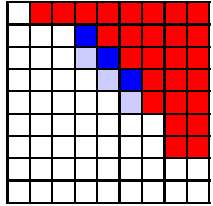
\[ \ytableausetup{boxsize=0.8em}
    \begin{ytableau}
      \, &*(red) & *(red)& *(red)& *(red)&*(red) &*(red) &*(red)&*(red)\\
      \, & & & *(blue)& *(red)&*(red) &*(red) &*(red)&*(red)\\
      &  & & *(blue!20) & *(blue) &*(red) &*(red)&*(red)&*(red)\\
      & &  & & *(blue!20)  & *(blue) &*(red)&*(red)&*(red)\\
      & & &  & & *(blue!20) & *(red)&*(red)&*(red)\\
      & & & &  & &&*(red) &*(red)\\
      & & & & &  &&*(red) &*(red)\\
      & & & & & &  &&\\
      & & & & & &  &&\\
    \end{ytableau}
  \]
\end{example}

\begin{lem}[Diagonal Removal Lemma]\label{big-staircase-lemma}
  Let \(\Psi \subset \Delta_\ell^+\) be a root ideal, \(M\) a multiset on
  \([\ell]\), \(\gamma \in \Z^\ell\), and integers $1 \leq x<y\leq
  z\leq\ell$ be such that
  \begin{enumerate}
  \item \(\Psi\) has a ceiling in columns \(z-1,z\) and every root of
    \(D_{x,y}^{z-1} \subset \Psi\) is removable from \(\Psi\);
  \item \(L(D_{x,y}^{z-1}) \subset M\) and \(m_M(z)=m_M(z-1) =
    m_M(z-2) + 1 = \cdots = m_M(y)+z-1-y\);
  \item \(\Psi\) has a wall in rows \(y,y+1, \dots, z\);
  \item \(\gamma_y = \cdots = \gamma_z\).
  \end{enumerate}
  Then, \[
    K(\Psi;M;\gamma) = K(\Psi';M';\gamma)
  \]
  where \(\Psi' = \Psi \setminus D_{x,y}^{z-1}\) and \(M' = M
  \setminus L(D_{x,y}^{z-1})\).
\end{lem}
\begin{proof}
Let $\beta^0,\beta^1,\ldots,\beta^{z-y-1}$ be the roots of the diagonal $D_{x,y}^{z-1}$
from lowest to highest, i.e., $\beta^j=(a_j,b_j)$ with  $a_j=z+x-y-j-1$ and $b_j=z-j-1$.
Define $\Psi^{j+1}=\Psi^{j}\setminus \{\beta^{j}\}$ and $M^{j+1}=M^{j}\setminus \{b_{j}\}$,
starting with $\Psi^0=\Psi$ and  $M^0=M$; thus $\Psi^{z-y-1}=\Psi'$ and $M^{z-y-1}=M'$.
By condition (a) for $j=0$ and by construction for $j>0$,
$\beta^{j}$ is a removable root of $\Psi^{j}$, and $\Psi^{j}$ has a
ceiling in columns $b_{j},b_{j}+1$.
Similarly, (b) implies that  $b_{j}\in M^{j}$ and $m_{M^j}(b_j+1) =m_{M^j}(b_j)$.
Therefore, using also (c) and (d), we can repeatedly apply
Lemma~\ref{staircase-lemma} to obtain
  \[
K(\Psi;M;\gamma) = K(\Psi^{1};M^{1};\gamma) =  K(\Psi^{2};M^{2} ;\gamma) = \cdots = K(\Psi^{z-y-1};M^{z-y-1};\gamma)\,.
\qedhere
  \]
\end{proof}

\subsection{Proof of Proposition~\ref{prop:tg-longest-word}}\label{proof-of-tg-longest-word}
  From $\Inv(w_0w_0)=0^k$ we have
 $\zeta(w_0)'=\big(\binom{k}{2},\ldots,\binom{1}{2}\big)=\theta(w_0)'$, and thus
 $\theta(w_0)=\union_{i=1}^{k-1} (k-i)^i$.
 The proposition states that  $\prod_{i=1}^{k-1} g_{(k-i)^i} = \tg_{\theta(w_0)^{\omega_k}}^{(k)}$, but we will first prove
 \begin{equation}
 \label{eq:prod-is-tg}
 \prod_{i=1}^{k-1} g_{(k-i)^i} = \tg_{\theta(w_0)}^{(k)} \,.
 \end{equation}

  Consider that, by Proposition~\ref{prop:extremal} and Lemma~\ref{katalan-concatenation},\[
    \prod_{i=1}^{k-1} g_{(k-i)^i} = \prod_{i=i}^{k-1} K(\emptyset_i;\emptyset_i;(k-i)^i) =
    K( \rootconcat_{i=1}^{k-1} \emptyset_i; \rootconcat_{i=1}^{k-1} \emptyset_i
    ; \textstyle \union_{i=1}^{k-1}(k-i)^i)
  \]
  where \(\emptyset_i \subset \Delta_i^+\) denotes the empty root ideal of
  length \(i\) and \(\rootconcat_{i=1}^{k-1} \emptyset_i = \emptyset_1 \rootconcat
  \emptyset_2 \rootconcat \cdots \rootconcat \emptyset_{k-1}\). Set
  \(\gamma = \union_{i=1}^{k-1}(k-i)^i\).
  We now proceed iteratively on \(i=1,\ldots,k-1\) with
  \(\Psi^{i} :=   \Delta^{(k)}(\union_{j=1}^{i} (k-j)^j) \rootconcat
  (\rootconcat_{j=i+1}^{k-1} \emptyset_j)\).
  For fixed \(i\), let \(a = 1+2+\cdots+i =
  \binom{i+1}{2}\). Note that \(\Psi^{i}\) has
  a ceiling in columns \(a+1, \ldots
  a+i+1\), a wall in rows \(a+1,\ldots,a+i+1\), and \(\gamma_{a+1} = \cdots = \gamma_{a+i+1} = k-i-1\). Now, we can apply
  Diagonal Removal Lemma~\ref{big-staircase-lemma} to
  \(K(\Psi^{i};\Psi^{i};\gamma)\) iteratively
  with \(x=a-d\), \(y=a+1\), and
  \(z=a+1+d\)
  for \(0 \leq d < i\) to get, for \(D_d = D_{a-d,a+1}^{a+1+d}\) and
  \(\Psi^i_d := \Psi^i \setminus (D_0 \union \cdots \union D_d)\),
    \begin{equation}
      \label{eq:iterate-sc-longest-word}
    K(\Psi^{i};\Psi^{i};\gamma) = K(\Psi^i_0 ; \Psi^i_0; \gamma) =
    K(\Psi^i_1 ;\Psi^i_1;\gamma) = \cdots = K(\Psi^i_{i-1} ;\Psi^i_{i-1};\gamma)
   = K(\Psi^{i+1};\Psi^{i+1};\gamma)\,,
   \end{equation}
  where the last equality follows since \(\Psi^i \setminus (D_0 \union
  \cdots \union D_{i-1})\) has \(i\) nonroots in rows
  \(a-i+1,\ldots,a\) and is thus equal to \(\Psi^{i+1}\). Then,~\eqref{eq:prod-is-tg}
  follows by applying~\eqref{eq:iterate-sc-longest-word} iteratively since \(\Psi^1 =
  \rootconcat_{j=1}^k \emptyset_j\) and \(\Psi^{k-1} = \Delta^{k}(\gamma)\).

By~\cite{morse}*{\S8}, there is an involution
\(\Omega \from \sym_{(k)} \to \sym_{(k)}\) defined by
\(\Omega(h_r) = g_{1^r}\), and it
satisfies $\Omega(g_\mu^{(k)}) =  g_{\mu^{\omega_k}}^{(k)}$
for all  $\mu \in \Par^k$.  Thus by Theorem \ref{formulations-are-equal}, $\Omega(\tg_\mu^{(k)}) =  \tg_{\mu^{\omega_k}}^{(k)}$ as well.  Applying  $\Omega$ to both sides of
\eqref{eq:prod-is-tg} yields
\begin{equation}
\prod_{i=1}^{k-1} \Omega(g_{(k-i)^i}) = \tg_{\theta(w_0)^{\omega_k}} \,.
\end{equation}
The left side is in fact equal to
$\prod_{i=1}^{k-1} g_{((k-i)^i)^{\omega_k}} = \prod_{i=1}^{k-1} g_{(i)^{k-i}} =
\prod_{i=1}^{k-1} g_{(k-i)^i}$,
where we have used that  $g_\mu = g_\mu^{(k)}$ and  $\mu^{\omega_k} = \mu'$ for
any  $\mu$ contained in a  $k$-rectangle, by Corollary \ref{k-rectangle-bounded-gk} and
\cite{LMtaboncores}*{Remark 10}.  This completes the proof.

\section{Vertical Pieri Rule}

We now apply the mirror lemmas to the $k$-Schur Katalan
functions, \(\g_\lambda^{(k)}\).  The root ideal combinatorics
matches naturally with previously studied $(k+1)$-core combinatorics
for \(g_\lambda^{(k)}\).
We deduce a vertical Pieri rule for
the \(\g_\lambda^{(k)}\), which agrees with the known rule for the \(g_\lambda^{(k)}\).

\subsection{Pieri straightening}

Recall from \eqref{e Deltak def} that $\Delta^k(\mu) = \{(i,j) \in \Delta^+_{\ell} \mid  k-\mu_i + i < j\}$.
This was defined for $\mu \in \Par^k_\ell$, but the definition can be extended to any
$\mu \in \ZZ_{\le k}^\ell$ such that $\mu_i \ge \mu_{i+1}-1$ for all  $i \in [\ell-1]$.
Several useful properties are satisfied by these $k$-Schur root ideals, immediate from their construction, which will
be used throughout this section.

\begin{rmk}\label{re:ksri}
  Let \(\lambda \in \Par^k_m\), \(\Psi =
  \Delta^k(\lambda)\), and  $\lowers = \Delta^{k+1}(\lambda)$.
   Let  $z$ be the lowest nonempty row of  $\Psi$.
    \begin{enumerate}
  \item \label{ksri:no-walls} (Wall-free) For  $x \in [z]$, \(\Psi\) does not have a wall in
    rows \(x,x+1\).  Hence for all  $x \in [m-1]$,
    either \(\Psi\) has a ceiling in
    columns \(x,x+1\) or has removable roots
    \((y,x)\) and \((y+1,x+1)\). In the latter case, if \(y \neq x-1\),
    then \(\Psi\) has a mirror in rows \(y,y+1\).
  \item \label{ksri:mirror-implies-equal-parts}
  (Equal weight mirrors) For  $x \in [z-1]$, \(\Psi\) has a mirror in rows \(x,x+1\) if and only if
  \(\mu_x = \mu_{x+1} < k\).
\item\label{ksri:no-lowering-walls}
(Wall-free lowering ideal)
For \(x \in [m-1]\), $\up_\Psi(x)$ exists $\iff m_{L(\lowers)}(x)=m_{L(\lowers)}(x+1)-1$. Otherwise $m_{L(\lowers)}(x)=m_{L(\lowers)}(x+1)$.
\item\label{ksri:last-part-flexible}
  (Adjustable end) Let  $S \subset \ZZ_{\ge m+2}$ satisfying $\max(S)-\min(S) \le k-1$ if it is nonempty.
  Set $\mu = \lambda + \epsilon_S \in \ZZ^\ell$ for  $\ell = \max(S \cup \{m\})$.
  Then $\Delta^k(\mu) = \Delta^k((\lambda,0^{\ell-m}))$ and
  $\Delta^{k+1}(\mu) = \Delta^{k+1}((\lambda,0^{\ell-m}))$, hence (a)--(c) apply with
   data $\ell, \mu,\Delta^k(\mu), \Delta^{k+1}(\mu)$ in place of
  $m, \lambda, \Psi, \lowers$.
\end{enumerate}
\end{rmk}
Here and throughout the remainder of the paper, for $\lambda\in\Par_\ell^k$ and $\alpha\in\mathbb Z^j$ with $j\geq \ell$, we define $\lambda+\alpha=(\lambda,0^{j-\ell})+\alpha$.

\begin{figure}[h]
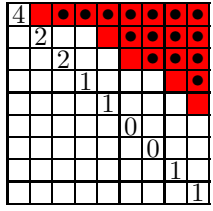

  \centering
    \[
    \ytableausetup{nobaseline}
      \begin{ytableau}
  4 & *(red) & *(red) \bullet & *(red)\bullet & *(red)\bullet & *(red)\bullet & *(red)\bullet & *(red)\bullet & *(red)\bullet \\
  {} & 2 & {} & {} & *(red) & *(red)\bullet & *(red)\bullet & *(red)\bullet & *(red)\bullet \\
  {} & {} & 2 & {} & {} & *(red) & *(red)\bullet & *(red)\bullet & *(red)\bullet \\
  {} & {} & {} & 1 & {} & {} & {} & *(red) & *(red)\bullet \\
  {} & {} & {} & {} & 1 & {} & {} & {} & *(red) \\
  {} & {} & {} & {} & {} & 0 & {} & {} & {} \\
  {} & {} & {} & {} & {} & {} & 0 & {} & {} \\
  {} & {} & {} & {} & {} & {} & {} & 1 & {} \\
  {} & {} & {} & {} & {} & {} & {} & {} & 1
\end{ytableau}
    \]
    \caption{\(K(\Psi;\lowers;\mu)\) with \(\mu = 422110011\),
    \(\Psi = \Delta^4(\mu)\) shown in red, and \(\lowers = \Delta^5(\mu)\)
    superimposed as \(\bullet\)'s as in Example~\ref{ex:katalan-conventions}.
    The nonzero row lengths of \(\Psi\) and \(\lowers\) decrease
    by at least one from top to bottom (illustrating (a) and (c)), and
    there are mirrors in rows $2,3$ and in rows  $4,5$ corresponding to  $\mu_2=\mu_3$ and
    $\mu_4=\mu_5$ (illustrating (b)). }
  \label{fig:k-schur-root-ideal}
\end{figure}

\begin{prop}[Pieri straightening]\label{pieri-straightening}
Let $\lambda\in\Par^k_m$ and $S\subset \ZZ_{\ge m+2}$ nonempty with $\max(S)-\min(S) \le k-1$.
Set
\(\mu = \lambda + \epsilon_{S}\),
\(\Psi= \Delta^{k}(\mu)\), \(M = L(\Delta^{k+1}(\mu))\),
and $j=\min(S)$. There holds
\begin{align}
\label{e pieri-straightening}
  K(\Psi;M \dunion S;\mu) =
  \begin{cases}
    K(\Delta^k(\nu);L(\Delta^{k+1}(\nu))\dunion(S \setminus j);\nu) & y=\top_\Psi(j-1) > \top_\Psi(j)\\
    -K(\Psi; M \dunion (S \setminus j); \mu-\epsilon_j) & \top_\Psi(j) > \top_\Psi(j-1)+1\\
    0 & \top_\Psi(j) = \top_\Psi(j-1)+1
  \end{cases}
\end{align}
where \(\nu := \mu+\epsilon_{\up_\Psi(y+1)}-\epsilon_j\) in the first case.
\end{prop}

\begin{figure}[h]
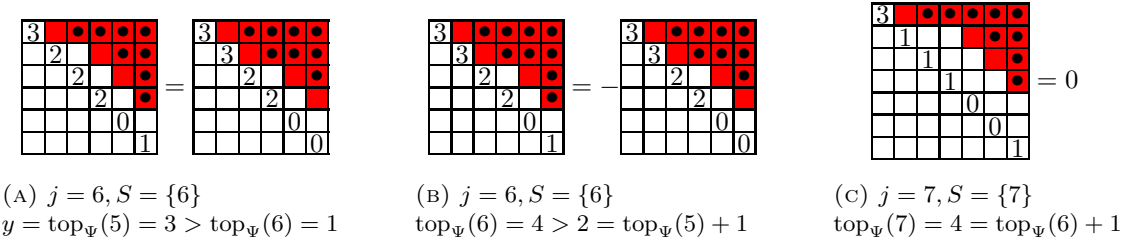
\ytableausetup{boxsize=0.8em,centertableaux}
  \centering
  \hspace{2.4mm}
  \begin{subfigure}[t]{0.31\linewidth}
    \[
    \ytableausetup{nobaseline}
    \hspace{-4mm}
       \begin{ytableau}
        3 & *(red) & *(red) \bullet & *(red) \bullet & *(red)\bullet & *(red)\bullet \\
        {} & 2 & {} & *(red) & *(red) \bullet& *(red) \bullet\\
        {} & {} & 2 & {} & *(red) & *(red)\bullet \\
        {} & {} & {} & 2 & {} & *(red)\bullet \\
        {} & {} & {} & {} & 0 & {} \\
        {} & {} & {} & {} & {} & 1 \\
      \end{ytableau}
      =
      \begin{ytableau}
        3 & *(red) & *(red) \bullet & *(red) \bullet & *(red)\bullet & *(red)\bullet \\
        {} & 3 & *(red) & *(red) \bullet & *(red) \bullet& *(red) \bullet\\
        {} & {} & 2 & {} & *(red) & *(red)\bullet \\
        {} & {} & {} & 2 & {} & *(red)\\
        {} & {} & {} & {} & 0 & {} \\
        {} & {} & {} & {} & {} & 0 \\
      \end{ytableau}
    \]
    \caption{\(j=6,S=\{6\}\\ y=\top_\Psi(5)=3 > \top_\Psi(6)=1\)}

  \end{subfigure}
  \hspace{2.4mm}
  \begin{subfigure}[t]{0.31\linewidth}
    \[
    \hspace{-3mm}
    \ytableausetup{nobaseline}
      \begin{ytableau}
        3 & *(red) & *(red)\bullet & *(red) \bullet& *(red) \bullet& *(red)\bullet \\
        {} & 3 & *(red) & *(red) \bullet& *(red)\bullet & *(red)\bullet \\
        {} & {} & 2 & {} & *(red) & *(red)\bullet \\
        {} & {} & {} & 2 & {} & *(red)\bullet \\
        {} & {} & {} & {} & 0 & {} \\
        {} & {} & {} & {} & {} & 1\\
      \end{ytableau}
      =
      -\begin{ytableau}
        3 & *(red) & *(red)\bullet & *(red) \bullet& *(red) \bullet& *(red)\bullet \\
        {} & 3 & *(red) & *(red) \bullet& *(red)\bullet & *(red)\bullet \\
        {} & {} & 2 & {} & *(red) & *(red)\bullet \\
        {} & {} & {} & 2 & {} & *(red)\\
        {} & {} & {} & {} & 0 & {} \\
        {} & {} & {} & {} & {} & 0\\
      \end{ytableau}
    \]
  \caption{\(j=6, S=\{6\}\\ \top_\Psi(6) = 4 > 2 = \top_\Psi(5)+1\)}
  \end{subfigure}
  \hspace{3.1mm}
  \begin{subfigure}[t]{0.27\linewidth}
      \vspace{-0.08cm}
    \[
    \hspace{-6mm}
    \ytableausetup{nobaseline}
      \begin{ytableau}
        3 & *(red) & *(red)\bullet & *(red)\bullet & *(red)\bullet & *(red)\bullet & *(red)\bullet \\
        {} & 1 & {} & {} & *(red) & *(red) \bullet& *(red)\bullet \\
        {} & {} & 1 & {} & {} & *(red) & *(red)\bullet \\
        {} & {} & {} & 1 & {} & {} & *(red)\bullet \\
        {} & {} & {} & {} & 0 & {} & {} \\
        {} & {} & {} & {} & {} & 0 & {} \\
        {} & {} & {} & {} & {} & {} & 1 \\
      \end{ytableau} = 0
    \]
    \caption{\(j=7, S=\{7\}\\ \top_\Psi(7) = 4 = \top_\Psi(6)+1\)}
  \end{subfigure}
  \caption{Examples of the three cases of
    Proposition~\ref{pieri-straightening} for \(k=3\).}
  \label{fig:pieri-straightening-cases}
\end{figure}
\begin{proof}
First, apply Proposition~\ref{root-expansions}(c) to \(j \in S\) to obtain
\begin{equation}
\label{Kex}
  K(\Psi;M \dunion S; \mu) = K(\Psi;M \dunion (S  \setminus j);\mu) - K(\Psi;M \dunion(S \setminus j);\mu-\epsilon_j)\,.
\end{equation}
Note that $\mu_{j-1}= \mu_j-1 = 0$ since \(\mu = \lambda + \epsilon_{S}\),
 $j = \min(S) \ge m+2$, and  $\lambda\in \Par^k_m$. Also, note throughout that, since \(\mu_{j-1}
= 0\), then \((j-1,j) \not \in \Psi\) and thus
\(\uppath_\Psi(j) \intersect \uppath_\Psi(j-1) = \emptyset\).

If \(y = \top_\Psi(j-1) > \top_\Psi(j)\), %
then \(\up_\Psi(y)\) does not exist but
\(\up_\Psi(y+1)\) does, so  $\Psi$ does not have a ceiling in
columns \(y,y+1\). Thus, Remark~\ref{re:ksri}
gives the conditions for Mirror Straightening Lemma~\ref{kk-schur-straightening} applied with \(z=j-1\)
to \(K(\Psi;M \dunion (S\setminus j);\mu)\) in~\eqref{Kex}, giving
\[K(\Psi;M \dunion (S\setminus j);\mu) = K(\Psi \union \alpha; M \dunion (y+1) \dunion (S \setminus j);
    \mu+\epsilon_{\up_\Psi(y+1)}-\epsilon_j)
+ K(\Psi;M \dunion
    (S \setminus j);\mu-\epsilon_j)\,,  \]
     where $\alpha=(\up_\Psi(y+1),y)$. Therefore,
\[
    K(\Psi;M \dunion S; \mu) =
K(\Psi \union \alpha; M \dunion (y+1) \dunion (S \setminus j);
    \mu+\epsilon_{\up_\Psi(y+1)}-\epsilon_j) \,.
  \]
Using \(\Psi \union \alpha = \Delta^{k}(\nu)\) and
\(M \dunion (y+1) = L(\Delta^{k+1}(\nu))\),
the top case of \eqref{e pieri-straightening} follows.

If \(\top_\Psi(j) > \top_\Psi(j-1)+1\),
then Remark \ref{re:ksri} gives the conditions
to apply Mirror Lemma~\ref{mirror-lemma}
with \(z=j-1\); note that, in this case,
there is no removable
root in column \(\top_\Psi(j)\) of \(\Psi\) by definition of \(\top\), but
there is a removable root of \(\Psi\) in column
\(\top_\Psi(j)-1\),  %
so \(\Psi\) has a ceiling in these columns.
In addition, \(m_{M}(\top_\Psi(j))-1 = m_{M}(\top_\Psi(j)-1)\) by Remark~\ref{re:ksri}(\ref{ksri:no-lowering-walls}),
so it is the first statement in
Mirror Lemma~\ref{mirror-lemma} that applies.
Hence the term \(K(\Psi;M \dunion(S \setminus j);\mu)\) in~\eqref{Kex}
vanishes, as desired.

If \(\top_\Psi(j)=\top_\Psi(j-1)+1\), then there are no removable roots
in columns \(\top_\Psi(j-1),\top_\Psi(j)\) of \(\Psi\) by definition of
\(\top\), so there is a ceiling in
columns \(\top_\Psi(j-1),\top_\Psi(j)\).
Remark~\ref{re:ksri} gives the conditions to apply Mirror Lemma~\ref{mirror-lemma}.
Since \(m_{L(\lowers)}(\top_\Psi(j-1)) =
m_{L(\lowers)}(\top_\Psi(j))\)
by Remark~\ref{re:ksri}(\ref{ksri:no-lowering-walls}), we obtain
\(K(\Psi;M \dunion(S \setminus j);\mu) = K(\Psi;M \dunion(S \setminus j);\mu-\epsilon_j)\), and thus the right side of \eqref{Kex} is
zero, as desired.
\end{proof}

\subsection{Katalan multiplication via root expansions}

Recall that $D^z_{x,y} \subset \Delta^+$ denotes the diagonal occupying columns  $y$ to  $z$, starting in row  $x$.
For $1\leq x<y\leq z$, a succession of diagonals, each occupying columns
$y$ to $z$, forms a {\it staircase}, $E^{z,h}_{x,y}=D_{x,y}^z\union D_{x+1,y}^z \union
\cdots\union D_{x+h-1,y}^z$.
In Example~\ref{ex:diagonals},  \(E_{2,4}^{6,2}\) is the union of light and dark blue cells.

\begin{lem}\label{general-pieri-lemma}
For \(\ell \geq 1\) and $r \geq 0$,
consider a root ideal \(\Psi \subset \Delta^+_{\ell+r}\) and a
multiset \(M\) on \([\ell+r]\).
Let $x,h\geq 0$ %
with  $x+r+h-2 \leq\ell$ be such that
  \begin{enumerate}
  \item \(E_{h} := E_{x,\ell+1}^{\ell+r,h}\subset\Psi\);
  \item \(\Psi' = \Psi \setminus E_{h}\) is a root ideal;
  \item \(m_M(\ell+1) \geq h\) and \(m_M(\ell+r) = m_M(\ell+r-1)+1 =
    \cdots = m_M(\ell+1) + r-1\).
  \end{enumerate}
Then, for \(\gamma \in \Z^\ell\) and \(M'  = M \setminus L(E_h)\),
\[
    K(\Psi; M; (\gamma,1^r)) =
\sum_{a=0}^r \sum_{\substack{
\mu = \gamma+\epsilon_{S} + \epsilon_{S'} \\
S \subset \{x+r-a, \ldots,
      x+r+h-2\} \\ |S| = a
\\ S' = \{\ell+1, \ldots, \ell+r-a\}}
}
    K(\Psi'; M' \dunion S; \mu)\,.
  \]
  where each summand is understood to be truncated in the manner of
  Remark~\ref{re:mismatched-lengths}.
\end{lem}
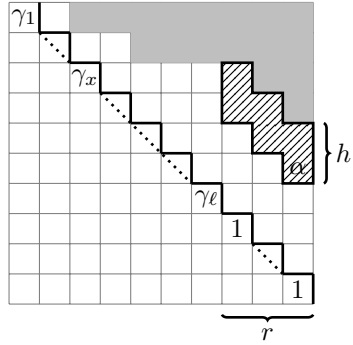
\begin{figure}[h]
  \centering
    \begin{tikzpicture}
    \draw[step=0.4cm, color=gray] (0, 0) grid (4, 4);

    \node at (0.2,3.8) {$\gamma_1$}; \draw[line width=1pt,
    style=dotted] (0.4,3.6) -- (0.8,3.2); \node at (1.0,3.0)
    {$\gamma_x$}; \draw[line width=1pt, style=dotted] (1.2,2.8) --
    (2.4,1.6); \node at (2.6,1.4) {$\gamma_\ell$}; \node at (3,1.0)
    {$1$}; \draw[line width=1pt, style=dotted] (3.6,0.4) -- (3.2,0.8);
    \node at (3.8,0.2) {$1$};

    \path[fill=lightgray]
    (0.8,4)--(0.8,3.6)--(1.6,3.6)--(1.6,3.2)--(2.0,3.2)--(3.2,3.2)
    --(3.2,2.8)--(3.6,2.8)--(3.6,2.4)--(4,2.4)--(4,4)--cycle;

    \draw[line width=1pt] \foreach \x in {1,...,9}{(0.4*\x,4-0.4*\x)
      -- (0.4*\x+0.4,4-0.4*\x)}; \draw[line width=1pt] \foreach \x in
    {0,1,...,9}{(0.4*\x+0.4,4-0.4*\x) -- (0.4*\x+0.4,3.6-0.4*\x)};

    \path[fill=white]
    (4,1.6)--(3.6,1.6)--(3.6,2.0)--(3.2,2.0)--(3.2,2.4)--(2.8,2.4)--(2.8,3.2)--(3.2,3.2)--(3.2,2.8)--(3.6,2.8)--(3.6,2.4)--(4,2.4)--cycle;

    \node at (3.8, 1.8) {$\alpha$};
    \draw[line width=1pt, pattern=north east lines]
    (4,1.6)--(3.6,1.6)--(3.6,2.0)--(3.2,2.0)--(3.2,2.4)--(2.8,2.4)--(2.8,3.2)--(3.2,3.2)--(3.2,2.8)--(3.6,2.8)--(3.6,2.4)--(4,2.4)--cycle;

    \draw [decorate,decoration={brace,mirror,amplitude=2pt,raise=4pt},
    yshift=0pt,line width=1pt] (4.0,1.6)--(4.0,2.4) node
    [midway,xshift=0.4cm] {$h$};

    \draw [decorate, decoration={brace, mirror,
      amplitude=2pt,raise=4pt}, yshift=0pt, line width=1pt]
    (2.8,0)--(4.0,0) node [midway, yshift=-0.4cm]{$r$};
  \end{tikzpicture}
  \caption{Schematic of the setup for Lemma~\ref{general-pieri-lemma} where
 $\Psi'$ are the roots in light gray, $E_h$ is the diagonally shaded region, and
    \(\Psi=\Psi'\cup E_h\).}
\end{figure}
\begin{proof}
If \(r = 0\) or \(h = 0\), \(E_{h}\) is the empty set and the equality
holds trivially. We proceed by induction on \(r+h\) with \(r,h>0\).
Noting that \(\alpha = (x+r+h-2, \ell+r)\) is the only root in the
lowest row of $E_h$, it is removable from $\Psi$ by (b).  Thus,
Lemma~\ref{sliding-lemma} implies
\[
K(\Psi;M;(\gamma,1^r))=
K(\Psi \setminus \alpha; M \setminus (\ell+r); (\gamma,1^r)) +
K(\hat{\Psi}; \hat{M} \dunion (x+r+h-2); (\gamma,1^{r-1}) +\epsilon_{x+r+h-2})\,.
\]
We shall apply Diagonal Removal Lemma~\ref{big-staircase-lemma} with $x=x+h-1, y=\ell+1,z=\ell+r$ to
the first term on the right hand side; %
indeed, $\Psi \setminus \alpha$ has a ceiling in $\ell+r-1,\ell+r$ and
(c) implies \(M \setminus (\ell+r)\) has the same number of occurrences of \(\ell+r-1,\ell+r\).
Furthermore, since $\Psi$ has no roots lower than $\alpha$, $\Psi\setminus\alpha$ has a wall in rows $x+r+h-2,\ldots,\ell+r$
(recall that $x+r+h-2\leq \ell$).
By definition, $D_{x+h-1,\ell+1}^{\ell+r}=E_h\setminus E_{h-1}$ is the lowest diagonal of $E_h$ and
thus every root of $D_{x+h-1,\ell+1}^{\ell+r-1}= D_{x+h-1,\ell+1}^{\ell+r}\setminus\alpha$ is removable from $\Psi\setminus\alpha$.
Therefore,
\[
K(\Psi;M;(\gamma,1^r))=
    K(\Psi' \union E_{h-1}; M' \dunion L(E_{h-1});(\gamma,1^r)) +
K(\hat{\Psi}; \hat{M} \dunion (x+r+h-2); (\gamma,1^{r-1}) +\epsilon_{x+r+h-2})\,.
  \]
The inductive hypothesis applied to the first term with \(h=h-1\) and
applied to the second term with \(r=r-1\) gives
  \begin{align*}
        K(\Psi; M; (\gamma,1^r))
    & = \sum_{a=0}^r \sum_{\substack{\mu =
      \gamma+\epsilon_{T} + \epsilon_{T'} \\ T \subset \{x+r-a, \ldots,
      x+r+h-3\} \\ |T| = a \\ T' = \{\ell+1, \ldots, \ell+r-a\}}}
  K(\Psi'; M' \dunion T; \mu) \\
    & + \sum_{a=0}^{r-1} \sum_{\substack{\mu =
      \gamma+\epsilon_{x+r+h-2}+\epsilon_{T} + \epsilon_{T'} \\ T
      \subset \{x+r-1-a, \ldots,
      x+r+h-3\} \\ |T| = a \\ T' = \{\ell+1, \ldots, \ell+r-1-a\}}}
    K(\Psi'; M' \dunion (T \union \{x+r+h-2\});\mu)
  \end{align*}
  Reindexing the second sum to go from  $1$ to  $r$ readily shows
  that we recover the desired sum, with the first sum corresponding
  to \(x+r+h-2 \not \in S\) and the second to \(x+r+h-2
  \in S\).
\end{proof}

\begin{prop}[Unstraightened Pieri Rule]\label{unstraightened-pieri-rule}
For $\lambda\in\Par^k_{\ell-k-1}$ and $0\leq r\leq k$,
  \[
    g_{1^r} \g_\lambda^{(k)} = \sum_{a=0}^{r}
\sum_{\substack{\mu = \lambda + \epsilon_{S} + \epsilon_{S'}\\
S \subset
      \{\ell-k+1+r-a,\ldots,\ell\} \\ |S| = a \\ S' = \{\ell+1,
      \ldots, \ell+r-a\}}}
    K(\Delta^{k}(\mu);L(\Delta^{k+1}(\mu))
    \dunion (S \union S'); \mu) \,.
  \]
\end{prop}
\begin{proof}
For $\lambda\in\Par^k_{\ell-k-1}$, Definition~\ref{def:KatKks} and Lemma~\ref{zero-lemma} give
\[
    \g_\lambda^{(k)} =
    K(\Delta^{k}((\lambda,0^{k+1}));\Delta^{k+1}((\lambda,0^{k+1}));
   (\lambda,0^{k+1}))\,.
 \]
Since \(g_{1^r} = K(\emptyset_r;\emptyset_r;1^r)\) by Proposition~\ref{prop:extremal}(b)
where \(\emptyset_r \subset \Delta^+_r\) denotes the empty root ideal of length \(r\), the concatenation rule of Lemma~\ref{katalan-concatenation} implies that
  \[
    g_{1^r} \g_{\lambda}^{(k)} =
K(\Psi,M,(\lambda,0^{k+1},1^r))\,,
  \]
for $\Psi = \Delta^{k}(\lambda,0^{k+1})\rootconcat \emptyset_r$ and $M= L(\Delta^{k+1}(\lambda,0^{k+1})\rootconcat \emptyset_r)$.

\begin{figure}[h]\centering
  \begin{subfigure}[t]{0.45\textwidth}
    \begin{tikzpicture}
      \draw[step=0.4cm, color=gray] (0, -0.4) grid (4.4, 4);

      \node at (0.2,3.8) {$\lambda_1$}; \draw[line width=1pt,
      style=dotted] (0.4,3.6) -- (0.8,3.2); \node at (1.0,3.0)
      {$\lambda_m$}; \node at (1.4,2.6) {$0$}; \node at (1.8,2.2)
      {$0$}; \draw[line width=1pt, style=dotted] (2.0,2.0) --
      (2.8,1.2); \node at (3,1) {$0$}; \node at (3.4,0.6) {$1$};
      \draw[line width=1pt, style=dotted] (3.6,0.4)--(4,0); \node at
      (4.2,-0.2) {$1$};

      \path[fill=lightgray]
      (3.2,4)--(4.4,4)--(4.4,0.8)--(3.2,0.8)--(3.2,2.8)--(2.4,2.8)--(2.4,3.2)--(1.6,3.2)--(1.6,3.6)--(0.8,3.6)--(0.8,4)--cycle;

      \foreach \x in {0,...,7}
        \node at (1.4+0.4*\x, 3.8) {$\bullet$};
      \foreach \x in {0,1,5}
        \node at (2.2+0.4*\x,3.4) {$\bullet$};
      \foreach \x in {0,...,3}
        \node at (3.0+0.4*\x,3.0) {$\bullet$};
      \foreach \x in {0,...,2}
        \foreach \y in {0,1,4}
          \node at (3.4+0.4*\x,2.6-0.4*\y) {$\bullet$};
      \node at (3.8,1.8) {$\vdots$};

      \node at (3.4,3.4) {$\Psi,M$};

      \draw[line width=1pt] \foreach \x in
      {1,...,10}{(0.4*\x,4-0.4*\x) -- (0.4*\x+0.4,4-0.4*\x)};
      \draw[line width=1pt] \foreach \x in
      {0,1,...,10}{(0.4*\x+0.4,4-0.4*\x) -- (0.4*\x+0.4,3.6-0.4*\x)};

      \draw[line width=1pt] (3.6,0.8) -- (4.4,0.8) -- (4.4,0);
      \path[fill=white] (3.6,0.8) -- (4.4,0.8) -- (4.4,0) -- cycle;
      \node at (4.1,0.6) {$\emptyset_r$};

      \draw [decorate, decoration={brace, mirror,
        amplitude=2pt,raise=4pt}, yshift=0pt, line width=1pt]
      (3.2,-0.4)--(4.4,-0.4) node [midway, yshift=-0.4cm]{$r$};

      \draw [decorate, decoration={brace, mirror,
        amplitude=2pt,raise=4pt}, yshift=0pt, line width=1pt]
      (1.2,-0.4)--(3.15,-0.4) node [midway, yshift=-0.4cm]{$k+1$};
    \end{tikzpicture}
  \end{subfigure}
  \begin{subfigure}[t]{0.45\textwidth}
    \begin{tikzpicture}
      \draw[step=0.4cm, color=gray] (0, -0.4) grid (4.4, 4);

      \node at (0.2,3.8) {$\lambda_1$}; \draw[line width=1pt,
      style=dotted] (0.4,3.6) -- (0.8,3.2); \node at (1.0,3.0)
      {$\lambda_m$}; \node at (1.4,2.6) {$0$}; \node at (1.8,2.2)
      {$0$}; \draw[line width=1pt, style=dotted] (2.0,2.0) --
      (2.8,1.2); \node at (3,1) {$0$}; \node at (3.4,0.6) {$1$};
      \draw[line width=1pt, style=dotted] (3.6,0.4)--(4,0); \node at
      (4.2,-0.2) {$1$};

      \path[fill=white]
      (3.2,2.4)--(4.4,2.4)--(4.4,0.8)--(3.2,0.8)--cycle;
      \path[fill=lightgray]
      (3.2,4)--(4.4,4)--(4.4,1.6)--(4.0,1.6)--(4.0,2.0)--(3.6,2.0)--(3.6,2.4)--(3.2,2.4)--(3.2,2.8)--(2.4,2.8)--(2.4,3.2)--(1.6,3.2)--(1.6,3.6)--(0.8,3.6)--(0.8,4)--cycle;
      \node at (3.4,3.4) {$\Psi'',M''$};

      \foreach \x in {0,...,7}
        \node at (1.4+0.4*\x, 3.8) {$\bullet$};
      \foreach \x in {0,1,5}
        \node at (2.2+0.4*\x,3.4) {$\bullet$};
      \foreach \x in {0,...,3}
        \node at (3.0+0.4*\x,3.0) {$\bullet$};
      \foreach \y in {0,1}
      {
        \pgfmathsetmacro{\z}{\y+1}
        \foreach \x in {\z,...,2}
        \node at (3.4+0.4*\x,2.6-0.4*\y) {$\bullet$};
      }
      \node at (3.8,1.8) {$\vdots$};

      \draw[line width=1pt] \foreach \x in
      {1,...,10}{(0.4*\x,4-0.4*\x) -- (0.4*\x+0.4,4-0.4*\x)};
      \draw[line width=1pt] \foreach \x in
      {0,1,...,10}{(0.4*\x+0.4,4-0.4*\x) -- (0.4*\x+0.4,3.6-0.4*\x)};

      \draw[line width=1pt] (3.6,0.8) -- (4.4,0.8) -- (4.4,0);
      \path[fill=white] (3.6,0.8) -- (4.4,0.8) -- (4.4,0) -- cycle;
      \node at (4.1,0.6) {$\emptyset_r$};

      \draw[line width=1pt,
      pattern=crosshatch]
      (3.2,0.8)--(4,0.8)--(4,1.2)--(3.6,1.2)--(3.6,1.6)--(3.2,1.6)--cycle;

      \draw[line width=1pt, pattern=north east lines]
      (4.4,0.8)--(4,0.8)--(4,1.2)--(3.6,1.2)--(3.6,1.6)--(3.2,1.6)--(3.2,2.4)--(3.6,2.4)--(3.6,2.0)--(4.0,2.0)--(4.0,1.6)--(4.4,1.6)--cycle;
      \node at (3.4,2.2) {$E$};

      \draw [decorate, decoration={brace, mirror,
        amplitude=2pt,raise=4pt}, yshift=0pt, line width=1pt]
      (3.2,-0.4)--(4.4,-0.4) node [midway, yshift=-0.4cm]{$r$};

      \draw [decorate, decoration={brace, mirror,
        amplitude=2pt,raise=4pt}, yshift=0pt, line width=1pt]
      (1.2,-0.4)--(3.15,-0.4) node [midway, yshift=-0.4cm]{$k+1$};

      \draw [decorate, decoration={brace, amplitude=2pt,raise=4pt}, yshift=0pt,
      line width=1pt](4.4,1.6)--(4.4,0.8) node [midway, xshift=1.0cm]{$k+1-r$};
    \end{tikzpicture}
  \end{subfigure}
    \caption{The schematic on the left represents $\Psi = \Delta^{k}(\lambda,0^{k+1})\rootconcat \emptyset_r$ and
$M = L(\Delta^{k+1}(\lambda,0^{k+1})\rootconcat \emptyset_r)$.  On the right,  $\Psi''$
and  $M''$ are the solid grey region and \(\bullet\)'s, respectively,
and the crosshatched region is $\Psi\backslash\Psi' = \Psi \setminus (\Psi''\cup E)$.  Here, $m=\ell-k-1$.}
\end{figure}
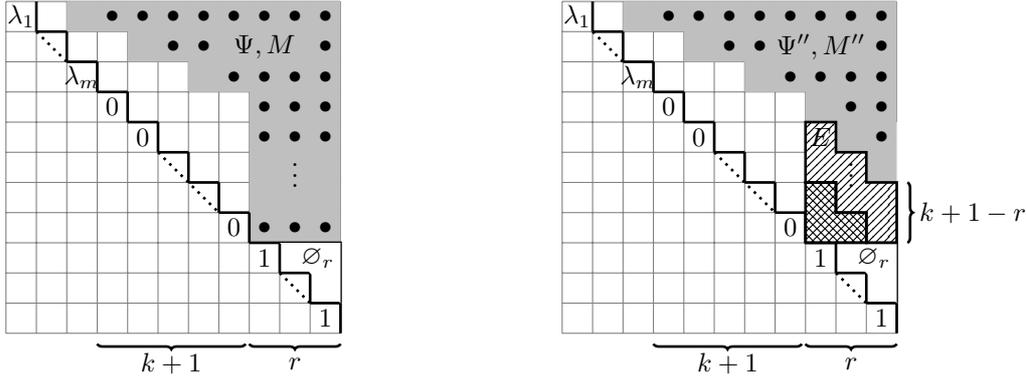
Let \(E = E_{\ell-k+1,\ell+1}^{\ell+r,k+1-r}\) and set
\begin{align*}
  \Psi''&= \Delta^{k}(\lambda,0^{k+1},1^r) \quad\text{ and $\quad\Psi'  = \Psi'' \union E\,;$}
\\
 M''&= L(\Delta^{k+1}(\lambda,0^{k+1},1^r)) \quad\text{
and $\quad M'  = M'' \dunion \{\ell+1,\ldots,\ell+r\} \dunion L(E)$.}
\end{align*}
Observe that $\Psi \setminus \Psi'=
D_{\ell,\ell+1}^{\ell+1}\cup D_{\ell-1,\ell+1}^{\ell+2}\cup
\cdots \cup D_{\ell-r+2,\ell+1}^{\ell+r-1}$ and \(M \setminus M' =
L(\Psi \setminus \Psi')\).
We remove these diagonals from $\Psi$ by iteratively applying Diagonal
Removal Lemma~\ref{big-staircase-lemma}
until
  \[
g_{1^r}\,\g_\lambda^{(k)} =
K(\Psi,M,(\lambda,0^{k+1},1^r))
    = K(\Psi';M';(\lambda,0^{k+1},1^r))\,.
  \]
We can then apply Lemma~\ref{general-pieri-lemma} with \(x = \ell-k+1, h = k+1-r,  \Psi = \Psi'\)
and \(M = M'\) to get
\[
  K(\Psi'; M'; (\lambda,0^{k+1},1^r)) =
\sum_{a=0}^r
\sum_{\substack{\mu =
      \lambda + \epsilon_{S} + \epsilon_{S'}\\
S \subset
      \{\ell+r-k+1-a,\ldots,\ell\} \\ |S| = a \\ S' = \{\ell+1,
      \ldots, \ell+r-a\}}}
      K(\Psi'';(M' \setminus L(E)) \sqcup S; \mu) \,.
  \]
Since $M' \setminus L(E) = M'' \dunion \{\ell+1,\ldots, \ell+r\}$,
we have for each summand $K(\Psi'';(M' \setminus L(E)) \sqcup S; \mu) = K(\Psi'';M'' \dunion (S \union S'); \mu)$ by Remark \ref{re:mismatched-lengths}.
Further, since \(\mu = \lambda+\epsilon_{S \cup S'}\) with  $\max(S\cup S')-\min(S\cup S') \le k-1$,
this is equal to $K(\Delta^k(\mu);L(\Delta^{k+1}(\mu)) \dunion (S \union S'); \mu)$
by Remark~\ref{re:ksri}(\ref{ksri:last-part-flexible}) and Remark \ref{re:mismatched-lengths}.
\end{proof}

\begin{lem}\label{lem:rm}
For \(\ell \geq 1\) and \(0 \leq r \leq k\), the map
  \[
{\rm rm}\from \Dunion_{a=0}^r
\Dunion_{
\substack{
S \subset
\{\ell-k+1+r-a,\ldots,\ell\} \\ |S| = a\\
S' = \{\ell+1, \ldots, \ell+r-a\} }
}
\{S\cup S'\}
\,
\to
\,
\big\{ R \subset \Z/(k+1)\Z \, : \,  |R|=r \big\}
\]
given by \(S\cup S' \mapsto \{\, \ov{-s}\,\st s \in S \union S'\}\) is a bijection, where $\ov{z}$ denotes the image of $z$ in  $\Z/(k+1)\Z$.
\end{lem}
\begin{proof}
For each $0\leq a\leq r$, $S\union S'$ is a subset
of the $k$ consecutive entries $\{\ell-k+1+r-a,\ldots, \ell+r-a\}$ and  $|S\union S'| = r$.
Thus \({\rm rm}\) is well-defined and one-to-one.
Given \(R \subset \Z/(k+1)\Z\) with \(|R|=r\),
to construct its preimage  $S \cup S'$, consider the largest \(b\) such that
$\{\ov{-(\ell+1)},\ldots,\ov{-(\ell+b)}\}\subset R$
or set $b=0$ if \(\ov{-(\ell+1)} \not \in R\).
Then  $S \cup S' = f_{\ell, b}(R)$, for the map  $f_{\ell,b} \colon \ZZ/(k+1)\ZZ
\to \ZZ$ given by
  \begin{align*}
    {\ov{-(\ell+i)}} & \ \ \ \mapsto  \ \ \ell+i   \, \ \text{ for } 1 \leq i
                     \leq b \\
    {\ov{-(\ell+b+j)}} & \ \ \ \mapsto \ \  \ell+b+j-k-1 \, \ \text{ for } 1 \le j \leq k+1-b.  \qedhere
  \end{align*}
\end{proof}

Combining Proposition \ref{unstraightened-pieri-rule} and Lemma~\ref{lem:rm} yields the following result.
\begin{cor}\label{unstraightened-pieri-rewrite}
For $\lambda\in\Par^k_\ell$ and $0\leq r\leq k$,
   \[
    g_{1^r} \g_\lambda^{(k)} = \sum_{\substack{R \subset
        \Z/(k+1)\Z \\ |R| = r}}
    K(\Delta^{k}(\lambda+\epsilon_{A}); L(\Delta^{k+1}(\lambda+\epsilon_{A})) \dunion A; \lambda+\epsilon_{A})
  \]
where $A={\rm rm}^{-1}(R)$.
\end{cor}

\subsection{Root ideal to core dictionary}\label{sec:ri-to-core-dict}

The diagram of a partition $\lambda$  is the subset of cells $\{(r,c)\in\mathbb Z_{\geq 1}\times
\mathbb Z_{\geq 1} : c\leq \lambda_r \}$ in the plane, drawn in English (matrix-style) notation so that
rows (resp. columns) are increasing from north to south (resp. west to east).
Each cell in a diagram has a \emph{hook length} which counts the
number of cells below it in its column and weakly to its right in its row.
An {\it $n$-core} is a partition with no cell of hook length $n$.
We use $\CC^{k+1}$ to denote the collection of $k+1$-cores.
There is a bijection~\cite{LMtaboncores},
$$
\partition: \CC^{k+1}\to \Par^k\,,
$$
where $\partition(\kappa) = \lambda$ is the partition whose $r$-th row, $\lambda_r$, is
the number of cells in the  $r$-th row of $\kappa$ with hook length $\le k$.
Let $\core=\partition^{-1}$.
The \emph{content} of a cell $(r,c) \in \Z \times \Z$ is $c-r$ and its
{\it $k+1$-residue} is $\ov{c-r} \in \ZZ/(k+1)\ZZ$.

Given a \(k+1\)-core \(\kappa\), define the row residue map
\[
\text{\rm r} \from \Z_{\geq 1} \to \ZZ/(k+1)\ZZ,\ \  a \mapsto \ov{ \kappa_a-a}\,,
\]
so that  $\text{\rm r}(a)$ is the $k+1$-residue of the cell $(a,\kappa_a)$ (if  $a \le \ell(\kappa)$, this lies on the eastern border of $\kappa$ but we also allow  $a > \ell(\kappa)$ and understand  $\kappa_a =0$ in this case).
We use the following lemma, obtained by taking~\cite{catalans}*{Proposition 8.2(b)} modulo \(k+1\).
\begin{lem}\label{row-residue-map-lem}
  Let \(\lambda \in \Par^k_\ell\) and \(\Psi = \Delta^{k}(\lambda)\).
  If \(\up_\Psi(x)\) is defined, then \( \text{\rm r}(\up_\Psi(x)) = \text{\rm r}(x) \).
\end{lem}

\begin{prop}\label{k-schur-root-ideal-has-kp1-bounce-paths}
Let \(\lambda \in \Par^k_\ell\) and $\kappa=\core(\lambda)$.
The root ideal \(\Delta^{k}(\lambda)\)  has at most
\(k+1\) distinct bounce paths and cells \((a,\kappa_a)\) and
\((b,\kappa_b)\)  have the same $k+1$-residue if and only
if $a$ and $b$ are in the same bounce path.
\end{prop}
\begin{proof}
Let \(\Psi = \Delta^{k}(\lambda,0^{k+1})\).
By construction, $\Psi$ has no roots in rows $[\ell+1,\ell+k+1]$,
implying that each of $\ell+1, \ldots,\ell+k+1$ lies in a distinct bounce path, $B_1,\ldots,B_{k+1}$, respectively.
Since $\down_\Psi(x)$ exists for all  $x \in [\ell]$,
$B_1,\ldots, B_{k+1}$ are the {\it only}
bounce paths in \(\Psi\).
Now, for $i\in[k+1]$, the $k+1$-residue of $(\ell+i,\kappa_{\ell+i})$ is
\begin{equation*}
\text{\rm r}(\ell+i)=\ov{\kappa_{\ell+i}-\ell-i}=\ov{0-(\ell+i)}\,.
\end{equation*}
Thus the residues $\text{\rm r}(\ell+1),\ldots,\text{\rm r}(\ell+k+1)$ are distinct and
so, by  Lemma~\ref{row-residue-map-lem},  $\text{\rm r}(a)=\text{\rm r}(i+\ell)$ for all $a\in B_i$.
Therefore, $\text{\rm r}(a)=\text{\rm r}(b)$ if and only if $a$ and $b$ lie in the same bounce path.
Because the bounce path of $x\in[\ell]$ in \(\Delta^{k}(\lambda)\) is a (possibly empty)
truncation of its bounce path in $\Psi$, the claim follows.
\end{proof}

Given a partition $\kappa$, an \emph{addable} $i$-corner is a cell $(r,c) \notin \kappa$ of  $k+1$-residue $i$
such that $\kappa\cup\{(r,c)\}$ is a partition;
a \emph{removable} $i$-corner is a cell $(r,c) \in \kappa$ of  $k+1$-residue $i$
such that $\kappa \setminus \{(r,c)\}$ is a partition.

\begin{prop}[\(K\)-\(k\)-Schur root ideal to core dictionary]\label{root-ideal-to-core-dict}
Let \(\lambda \in \Par^k_j\) with  $\lambda_{j-1} = \lambda_j = 0$. Set
$i = \ov{-j+1}$. Then the bounce paths of  \(\Psi = \Delta^{k}(\lambda)\) are related to the
 $k+1$-core $\kappa=\core(\lambda)$ as follows.
Also, (a)--(c) below hold more generally with
root ideal  \(\Delta^{k}(\lambda+\epsilon_S)\)
in place of $\Psi$, for any  $S \subset \Z_{\ge j}$.
 \begin{enumerate}
\item \(y=\top_\Psi(j-1)>\top_\Psi(j)\) if and only if
the lowest addable  $i$-corner of $\kappa$ lies in row $a=\up_\Psi(y+1)$,

\item
\(  \top_\Psi(j)>\top_\Psi(j-1)+1\)
  if and only if \(\kappa\) has a removable  $i$-corner,

\item
\(\top_\Psi(j) = \top_\Psi(j-1)+1\)
    if and only if
    \(\kappa\) has neither a removable  $i$-corner nor addable  $i$-corner.
\end{enumerate}
\end{prop}
\begin{proof}
Let  $\text{\rm r}$ be the row residue map of  $\kappa$.
Noting $\text{\rm r}(j) = i-1$ and  $\text{\rm r}(j-1) = i$,  by Proposition~\ref{k-schur-root-ideal-has-kp1-bounce-paths}, the set of row indices   $\{z \in [j] \mid \text{\rm r}(z) = i-1\} = \uppath_\Psi(j)$
and
$\{z \in [j] \mid \text{\rm r}(z) = i\} = \uppath_\Psi(j-1)$.
Using this, we have
\begin{align}
\notag
&\text{$\kappa$ has an addable  $i$-corner in row  $z \in [j]$ \ \,  $\iff$  \ \,  $\text{\rm r}(z-1) \ne i$ and
$\text{\rm r}(z) = i-1$}  \\
\label{e rootidealtocore 1}
& \text{$\iff$  \ \, $z-1 \notin \uppath_\Psi(j-1)$ and $z \in \uppath_\Psi(j)$}\,; \\
\notag
&\text{$\kappa$ has a removable $i$-corner in row  $z-1 \in [j]$ \ \,  $\iff$  \ \,  $\text{\rm r}(z-1) = i$ and  $\text{\rm r}(z) \ne i-1$}
\\
\label{e rootidealtocore 2}
& \text{$\iff$  \ \, $z-1 \in \uppath_\Psi(j-1)$ and $z \notin \uppath_\Psi(j)$.}
\end{align}

For $y=\max\{\top_\Psi(j-1),\top_\Psi(j)-1\}$, since \(j\) and \(j-1\)
cannot be in the same bouncepath, $\Psi$ has a mirror in rows $x,x+1$
for \(x \in \uppath_\Psi(\up_\Psi(j-1))\) such that \(x \geq y\) by Remark~\ref{re:ksri}(\ref{ksri:no-walls}).
Thus, the bounce paths $\uppath_\Psi(j-1)$ and $\uppath_\Psi(j)$ have one of the following forms:
(a) $j-1, j_2-1,j_3-1,\dots,y$ and  $j,j_2,j_3,\dots, y+1,a,\dots$;
(b) $j-1, j_2-1,\dots,y,b,\dots$ and  $j,j_2,\dots, y+1$;
or (c) $j-1, j_2-1,\dots,y$ and  $j,j_2,\dots, y+1$.
The result now follows from \eqref{e rootidealtocore 1}--\eqref{e rootidealtocore 2}.
Note that the more general statement
holds simply because  $\uppath_{\Delta^k(\lambda+\epsilon_S)}(j-1) = \uppath_{\Psi}(j-1)$ and $\uppath_{\Delta^k(\lambda+\epsilon_S)}(j) = \uppath_{\Psi}(j)$.
\end{proof}

\begin{example}
  For \(k=5\) and \(\lambda = 532222111100000\), set \(\Psi =
  \Delta^{5}(\lambda)\) and \(\lowers = \Delta^{6}(\lambda)\). Then,
  \[
    \ytableausetup{boxsize=0.8em,centertableaux,nobaseline}
    \kappa = {\mathfrak c}(\lambda) =
    \begin{ytableau}
    0 & 1 & 2 & 3 & 4 & 5 & 0 & 1 & 2 & 3 & 4 \\
    5 & 0 & 1 & 2 & 3 & 4  \\
    4& 5 & 0  \\
    3& 4& 5  \\
    2& 3& 4 \\
    1& 2& 3 \\
    0 \\
    5 \\
    4\\
    3\\
    \none \\
    \none \\
    \none \\
    \none \\
    \none \\
    \end{ytableau}
    \ \
      \begin{ytableau}
  5 & *(red) & *(red) \bullet & *(red) \bullet & *(red) \bullet & *(red) \bullet & *(red) \bullet & *(red) \bullet & *(red) \bullet & *(red) \bullet & *(red) \bullet & *(red) \bullet & *(red) \bullet & *(red) \bullet & *(red) \bullet \\
  {} & 3 & {} & {} & *(red) & *(red) \bullet & *(red) \bullet & *(red) \bullet & *(red) \bullet & *(red) \bullet & *(red) \bullet & *(red) \bullet & *(red) \bullet & *(red) \bullet & *(red) \bullet \\
  {} & {} & 2 & {} & {} & {} & *(red) & *(red) \bullet & *(red) \bullet & *(red) \bullet & *(red) \bullet & *(red) \bullet & *(red) \bullet & *(red) \bullet & *(red) \bullet \\
  {} & {} & {} & 2 & {} & {} & {} & *(red) & *(red) \bullet & *(red) \bullet & *(red) \bullet & *(red) \bullet & *(red) \bullet & *(red) \bullet & *(red) \bullet \\
  {} & {} & {} & {} & 2 & {} & {} & {} & *(red) & *(red) \bullet & *(red) \bullet & *(red) \bullet & *(red) \bullet & *(red) \bullet & *(red) \bullet \\
  {} & {} & {} & {} & {} & 2 & {} & {} & {} & *(red) & *(red) \bullet & *(red) \bullet & *(red) \bullet & *(red) \bullet & *(red) \bullet \\
  {} & {} & {} & {} & {} & {} & 1 & {} & {} & {} & {} & *(red) & *(red) \bullet & *(red) \bullet & *(red) \bullet \\
  {} & {} & {} & {} & {} & {} & {} & 1 & {} & {} & {} & {} & *(red) & *(red) \bullet & *(red) \bullet \\
  {} & {} & {} & {} & {} & {} & {} & {} & 1 & {} & {} & {} & {} & *(red) & *(red) \bullet \\
  {} & {} & {} & {} & {} & {} & {} & {} & {} & 1 & {} & {} & {} & {} & *(red) \\
  {} & {} & {} & {} & {} & {} & {} & {} & {} & {} & 0 & {} & {} & {} & {} \\
  {} & {} & {} & {} & {} & {} & {} & {} & {} & {} & {} & 0 & {} & {} & {} \\
  {} & {} & {} & {} & {} & {} & {} & {} & {} & {} & {} & {} & 0 & {} & {} \\
  {} & {} & {} & {} & {} & {} & {} & {} & {} & {} & {} & {} & {} & 0 & {} \\
  {} & {} & {} & {} & {} & {} & {} & {} & {} & {} & {} & {} & {} & {} & 0
\end{ytableau}  = K(\Psi;\lowers;\lambda)
    \,,
  \]
  where we have filled the cells of $\kappa$ with their  $k+1$-residues.
  Note that, for example, \(\ov{4}=\text{\rm r}(1)=\text{\rm
    r}(2)=\text{\rm r}(5)=\text{\rm r}(9)=\text{\rm r}(14)\)
  illustrating Lemma~\ref{row-residue-map-lem}. We can
  also observe examples of all three cases of Proposition~\ref{root-ideal-to-core-dict}.
  For (a), let \(j=14\). Then, \(\top_\Psi(14) = 1 < 4 = \top_\Psi(13)\)
  and the lowest addable corner of residue \(i = \ov{-14+1} = \ov{5}\)
  is in row \(2\) of \(\kappa\). For (b), let \(j=15\). Then,
  \(\top_\Psi(15) = 6 > 1+1 = \top_\Psi(14)+1\) and \(\kappa\) has a
  removable corner of residue \(i = \ov{-15+1} = \ov{4}\). Finally,
  for (c), let \(j=13\). Then, \(\top_\Psi(13) = 4 = \top_\Psi(12)+1\)
  and \(\kappa\) has neither a removable nor an addable corner of
  residue \(\ov{-13+1} = \ov{0}\).
\end{example}

\begin{lem}[\cite{LMtaboncores}*{Proposition 22, \S8.1}]\label{le:addweak}
Let  $\kappa$ be a \(k+1\)-core and
\(\lambda = \partition(\kappa)\).
Then  $s_i w_\lambda \in \eS_{k+1}^0$
if and
only if  $\kappa$ has an addable or removable  $i$-corner.
Moreover, \(\kappa\) has an addable \(i\)-corner if and
only if \(s_i w_\lambda = w_{\lambda+\epsilon_a} \in \eS_{k+1}^0\), where  $a$ is the row index of the
lowest addable \(i\)-corner of \(\kappa\).
The  core \(\kappa\) has a removable \(i\)-corner if and
only if \(s_i w_\lambda = w_{\lambda-\epsilon_a} \in \eS_{k+1}^0\),
where  $a$ is the row index of the lowest removable \(i\)-corner of \(\kappa\).
\end{lem}

\subsection{Proof of the vertical Pieri rule}

\begin{prop}\label{prop:band-k-katalans-k-bdd-subspace}
For a root ideal \(\Psi \subset \Delta^+_\ell\),
a multiset \(M\) on  \([\ell]\), and
\(\gamma \in \Z^\ell\) satisfying
maxband$(\Psi,\gamma)~\leq~k$, there holds
\(K(\Psi;M;\gamma) \in \sym_{(k)}\).  Here, maxband is as defined in \eqref{ed maxband}.
\end{prop}
\begin{proof}
  Consider that, by definition, \[
    K(\Psi;M;\gamma) = \sum_{A \subset M} (-1)^{|A|} H(\Psi;\gamma - \epsilon_A)\,,
  \]
  where the summation is over all sub-multisets \(A\) of \(M\). Since maxband\((\Psi,\gamma-\epsilon_A) \leq k\),
   each summand
  \(H(\Psi;\gamma-\epsilon_A) \in \sym_{(k)}\) by~\cite{ksplitcatalans}*{Proposition 1.4}.
\end{proof}

\begin{prop}
\label{prop:triangular}
The set \(\{\g_\lambda^{(k)}\}_{\lambda \in \Par^k}\) forms a basis for \(\sym_{(k)}\).  Moreover,
it is unitriangularly related to the  $k$-Schur basis, i.e.,
  \(\g_\lambda^{(k)} = s_\lambda^{(k)} + \sum_{|\mu|<|\lambda|}
a_{\lambda \mu} s_\mu^{(k)}\) for \(a_{\lambda \mu} \in \Z\).
\end{prop}
\begin{proof}
By Proposition~\ref{prop:band-k-katalans-k-bdd-subspace}, \(\g_\lambda^{(k)}\)
lies in \(\sym_{(k)}\) and so can be written in terms of the $k$-Schur basis of $\Lambda_{(k)}$;
this expansion has the stated form since the highest degree term of \(K(\Psi;M;\gamma)\) is \(H(\Psi;\gamma)\)
irrespective of $M$.  Hence the transition matrix from \(\{\g_\lambda^{(k)}\}\) to  $\{s_\mu^{(k)}\}$ is unitriangular and thus the former is a basis.
\end{proof}

Recall from Section~\ref{section 2} that  $w_\lambda\in\eS^0_{k+1}$ is the minimal coset representative corresponding to  $\lambda\in \Par^k$.
For any $\lambda\in \Par^k$, set \(\g_{w_\lambda}^{(k)} = \g_\lambda^{(k)}\),
so that the basis  $\{\g_\lambda^{(k)}\}_{\lambda \in \Par^k}$ can also be written
$\{\g_v^{(k)}\}_{v \in \eS^0_{k+1}}$.
Recall that $H_{k+1}$ denotes the \(0\)-Hecke algebra of  $\eS_{k+1}$ with generators  $\{T_i \mid i \in \{0,1,\dots, k\}\}$.

\begin{prop}\label{prop:hecke-action}
    The rule
  \begin{align}\label{e siong}
    T_i \cdot \g_{v}^{(k)} =
    \begin{cases}
      \g_{s_i v}^{(k)} & \ell(s_i v) > \ell(v) \text{ and } s_iv \in
      \eS_{k+1}^0\,, \\
      -\g^{(k)}_{v} & \ell(s_i v) < \ell(v)\,, \\
      0 & s_i v \not \in \eS_{k+1}^0\,,
    \end{cases}
\end{align}
for  $i \in \{0,1,\dots, k\}$ and  $v \in \eS^0_{k+1}$,
determines an action of $H_{k+1}$ on \(\sym_{(k)}\).
\end{prop}
Note that the three cases are mutually exclusive since  $\ell(s_i v) < \ell(v)$ implies
$s_i v \in \eS_{k+1}^0$.
\begin{proof}
  Consider \(e = \sum_{w \in S_{k+1}} T_w \in H_{k+1}\) and note
  that for \(i \in [k]\), \(T_i e = 0\) and so \(T_u e = 0\) for any
  \(u \in \eS_{k+1} \setminus \eS_{k+1}^0\).
  Recalling that $\{T_w\}_{w \in \eS_{k+1}}$ is a  $\Z$-basis of  $H_{k+1}$,
  it follows that the left module \(M = H_{k+1} e \) has \(\Z\)-basis \(\{T_v e\}_{v \in \eS_{k+1}^0}\).
  We then check that the
  \(\Z\)-linear map \(M \to \sym_{(k)}\) given by \(T_v e \mapsto \g_v^{(k)}\) is an  $H_{k+1}$-module isomorphism
  by computing
  \[ T_i \cdot T_v e= T_iT_v e
    =\begin{cases}
      T_{s_i v}e    &
      \ell(s_i v) > \ell(v) \text{ and } s_iv \in \eS_{k+1}^0\,,\\
      -T_ve   & \ell(s_i v) < \ell(v)\,,  \\
      T_{s_i v}e = 0 & \ell(s_i v) > \ell(v) \text{ and } s_i v \not\in \eS_{k+1}^0\,.
   \end{cases}
  \]
\end{proof}

\begin{lem}\label{si-operator-thm}
For $\lambda\in\Par^k_m$, $0\leq r\leq k$, and $S=\{a_1<a_2<\cdots<a_r\}\subset \ZZ_{\ge m+2}$ with $a_r-a_1 \le k-1$,
 \[
    K(\Delta^{k}(\mu); L(\Delta^{k+1}(\mu)) \dunion S;
    \mu) = T_{i_r} \cdots T_{i_1} \g_{w_\lambda}^{(k)}\,,
  \]
where $\mu=\lambda+\epsilon_S$ and \(i_z := \ov{-a_z+1} \)
for $z \in [r]$.
\end{lem}

\begin{proof}
If $|S|=0$, then the claim holds by definition of $\g_{w_\lambda}^{(k)}$.
Proceed by induction, with \(|S| = r>0\).
Set $\kappa=\core(\lambda)$, $\Psi=\Delta^{k}(\mu)$, and $M = L(\Delta^{k+1}(\mu))$.
Let \(j= a_1 = \min (S)\), and note \(i_1 = \ov{-j+1}\).

First suppose \(y =\top_\Psi(j-1) > \top_\Psi(j)\). Then Proposition~\ref{pieri-straightening}
implies
\begin{align*}
      K(\Psi; M \dunion S; \mu)
      & = K(\Delta^{k}(\nu); L(\Delta^{k+1}(\nu)) \dunion (S \setminus j); \nu) \,,
\end{align*}
for \(\nu := \mu+\epsilon_a - \epsilon_j\), where \(a = {\up_\Psi(y+1)}\).
Since \(\nu= (\lambda+\epsilon_a) +\epsilon_{S\backslash \{j\}}\), induction gives
\[
 K(\Delta^{k}(\nu); L(\Delta^{k+1}(\nu)) \dunion (S \setminus j); \nu)
       = T_{i_r} \cdots T_{i_2} \g_{\lambda+\epsilon_a}^{(k)}\,.
\]
By Proposition~\ref{root-ideal-to-core-dict}(a),
the lowest addable $i_1$-corner of $\kappa$ lies in row $a$.
Therefore, $w_{\lambda+\epsilon_a} = s_{i_1} w_\lambda$ by Lemma~\ref{le:addweak}.
Then, by Proposition~\ref{prop:hecke-action} and
the fact that
$\ell(w_\lambda) = |\lambda|$, we have
$ \g_{\lambda+\epsilon_a}^{(k)}
= \g_{s_{i_1} w_\lambda}^{(k)} = T_{i_1} \g_\lambda^{(k)}$.

Next suppose \(\top_\Psi(j) > \top_\Psi(j-1)+1\). Proposition~\ref{pieri-straightening} yields
\[
K(\Psi; M \dunion S; \mu)
=
-K(\Psi; M \dunion (S \setminus \{j\}); \lambda+ \epsilon_{S \setminus \{j\}}).
\]
Rewriting using Remark~\ref{re:ksri}(\ref{ksri:last-part-flexible}) with  $\nu = \lambda+ \epsilon_{S \setminus \{j\}}$, and then applying induction yields
\[
-K(\Psi;M \dunion (S \setminus \{j\}); \lambda+ \epsilon_{S \setminus \{j\}})
= -K(\Delta^{k}(\nu); L(\Delta^{k+1}(\nu)) \dunion (S \setminus \{j\}); \nu)
 = -T_{i_r} \cdots T_{i_2} \g_{\lambda}^{(k)}\,.
\]
By Proposition~\ref{root-ideal-to-core-dict}(b), $\kappa$ has a removable \(i_1\)-corner, so
$-\g_\lambda^{(k)} =T_{i_1} \g_\lambda^{(k)}$ by Lemma~\ref{le:addweak} and Proposition~\ref{prop:hecke-action}.

Finally, suppose \(\top_\Psi(j) = \top_\Psi(j-1)+1\).  Proposition~\ref{pieri-straightening} yields
\(K(\Psi; M \dunion S; \mu) = 0\).
By Proposition~\ref{root-ideal-to-core-dict}(c), $\kappa$ has neither an addable nor a removable \(i_1\)-corner,
so
$T_{i_r} \cdots T_{i_2} T_{i_1} \g_\lambda^{(k)}=T_{i_r} \cdots T_{i_2} \big(T_{i_1} \g_\lambda^{(k)}\big)
=0$ by Lemma~\ref{le:addweak} and Proposition~\ref{prop:hecke-action}.
\end{proof}

We can now complete the proof of Theorem~\ref{formulations-are-equal} by showing the  $\g_\lambda^{(k)}$ satisfy the Pieri rule \eqref{gpieri}.
\begin{thm}\label{thm:katalan-kschur-pieri}
  For \(0 \leq r \leq k\) and \(\lambda \in \Par^k\),
  \[
    g_{1^r} \g_\lambda^{(k)} =
    \sum_{\substack{u \in \eS_{k+1} \text{ cyclically increasing} \\ \ell(u) = r
    \\ T_u T_{w_\lambda} = \pm T_w ; \, w \in \tilde{S}_{k+1}^0}}
    (-1)^{\ell(w_\lambda)+r-\ell(w)}\g_{w}^{(k)}\,.
  \]
\end{thm}
\begin{proof}
Corollary~\ref{unstraightened-pieri-rewrite} gives
$$
    g_{1^r} \g_\lambda^{(k)}
     = \sum_{\substack{R \subset \Z/(k+1)\Z \\ |R|=r}} K(\Delta^{k}(\mu); L(\Delta^{k+1}(\mu)) \dunion A; \mu)\,,
$$
where $\mu=\lambda+\epsilon_A$ for $A=\text{rm}^{-1}(R)$.
The result then follows by applying Lemma~\ref{si-operator-thm} to each
summand to get
$$
    g_{1^r} \g_\lambda^{(k)}
     = \sum_{s_{i_r} \cdots  s_{i_1} \text{ cyclically increasing}} T_{i_r} \cdots T_{i_1} \g_{w_\lambda}^{(k)}
$$
and then using Proposition~\ref{prop:hecke-action}.%
\end{proof}

\section{Appendix}\label{appendix}
\subsection{Raising operator identity for dual stable Grothendieck polynomials:
proof of~\eqref{eq:dual-grothendieck-jacobi-trudi}}\label{proof-of-dual-grothendieck-jacobi-trudi}
  By the proof of~\cite{macdonald}*{I. (3.4'')}, the following identity holds in the ring \(\mathbb{A} = \Z\llbracket \frac{z_1}{z_2}, \ldots,  \frac{z_{\ell-1}}{z_\ell}\rrbracket [z_1, \ldots, z_\ell]\),
  \[
    \sum_{w \in S_\ell} (-1)^{\ell(w)} \mathbf{z}^{\gamma + \rho - w \rho} =
    \prod_{i < j} \Big(1 - \frac{z_i}{z_j}\Big) \mathbf{z}^\gamma
  \]
  where \(\rho = (\ell-1,\ell-2,\ldots,1,0)\).  Applying the map
  \(\kappa\) from \eqref{eq:formal2} then
  yields  $ \prod_{i < j} \big(1 - R_{ij}\big) k_\gamma$ on the right and
  $\det(k_{\gamma_i+\rho_i - \rho_j}^{(i-1)})_{1 \leq i,j \leq \ell} = \det(k_{\gamma_i+j-i}^{(i-1)})_{1 \leq i,j \leq \ell} = g_\gamma$ on the left, thus establishing \eqref{eq:dual-grothendieck-jacobi-trudi}.

\subsection{Proof of \(G_{1^m}^{(k)} = G_{1^m}\)}\label{sec:GisGk}
By~\cite{lss}*{\S7.4}, the coefficient of the
monomial symmetric function
\(m_\mu\) in \(G_{\lambda}^{(k)}\) for \(\lambda \in \Par^k\) and \(\mu\) a partition of length \(a=\ell(\mu)\)
 is equal to  $(-1)^{|\mu|-|\lambda|}$ times the number of factorizations
\( T_{w_\lambda} = \pm T_{u_1} \cdots T_{u_a}\) in the 0-Hecke algebra \(H_{k+1}\),
for cyclically decreasing words  $u_1,\dots, u_a$ of lengths \(\mu_1, \dots, \mu_a\).
We have
\(w_{1^m} = s_{-m+1} \cdots s_{-1} s_0\) with indices taken modulo
\(k+1\). Since no braid or commutations relations can be applied to this word,
the only factorizations of \(T_{w_{1^m}}\) of the above form are
\[   T_{w_{1^m}} = \pm T_{-m+1}^{a_m} \cdots T_{-1}^{a_2}
  T_{0}^{a_1} \,, \ \ a_i \geq 1 \,,
\]
each  $u_i$ being a simple reflection.
Thus the coefficient of  $m_\mu$ in  $G_{1^m}^{(k)}$ is 0 unless
\(\mu=(1^{a})\) for \(a \geq m\). For each such \(a\), there are exactly
\(\binom{a-1}{m-1}\) possible factorizations. Therefore
 \[
  G_{1^m}^{(k)} = \sum_{a \geq m}
   (-1)^{a-m} \binom{a-1}{m-1} m_{1^a} =
   \sum_{i \geq 0} (-1)^i \binom{m+i-1}{m-1} e_{m+i}
   = G_{1^m} \,,
\]
where the last equality is a well-known formula
for \(G_{1^m}\) (which we used earlier in Theorem~\ref{shift-invariance}).

\subsection{Equivalence of \(K\)-\(k\)-Schur function descriptions}\label{sec:eq-of-Kks}

\begin{rmk}
The affine stable Grothendieck polynomials  $\{G_\mu^{(k)}\}_{\mu \in \Par^k}$
and $K$-$k$-Schur functions $\{g_\lambda^{(k)}\}_{\lambda \in \Par^k}$
are presented somewhat differently in \cite{lss} and \cite{morse},
but are indeed the same.
For the  $G_\mu^{(k)}$'s, this is by \cite[\S7.4]{lss}
and \cite[(31)--(32) and Theorem 28]{morse}.
Moreover, in both papers,
the  $G_\mu^{(k)}$'s and  $g_\lambda^{(k)}$'s determine each other
by  $\langle  g_\lambda^{(k)}, G^{(k)}_\mu \rangle = \delta_{\lambda \mu}$
(see \cite[\S7.5]{lss} and \cite[Property 40]{morse}).
\end{rmk}

\begin{proof}[Proof of Theorem \ref{t K k Schur basics}]
By \cite{lss}*{Theorems 6.8 and 7.17(1)},
there are Hopf algebra isomorphisms
$K_*({\rm Gr_{SL_{k+1}}}) \to \mathbb{L}_0 \to \sym_{(k)}$
under which $\xi^0_{w} \mapsto \varphi_0(k_{w}) \mapsto g_w^{(k)}$ for all $w\in \eS_{k+1}^0$,
where the $\varphi_0(k_{w})$ are versions of  $K$-$k$-Schur functions lying in
a subalgebra  $\mathbb{L}_0$ of the 0-Hecke algebra  $H_{k+1}$.
Equation~(6.1) and Corollary~7.6 of \cite{lss} determine certain structure constants of the
$\varphi_0(k_{w})$; the $g_w^{(k)}$ have the same structure constants, so translating notation from \cite{lss} gives
that for all  $v \in \eS_{k+1}^0$ and  $r \in [k]$,
\begin{equation}
  \label{eq:horiz-pieri}
  g_{s_{r-1} \cdots s_0}^{(k)} \, g_v^{(k)} = \sum_{\substack{u \in \eS_{k+1} \text{\,cyclically decreasing} \\
  \ell(u) = r \\ T_uT_v = \pm T_w; \, w \in \eS_{k+1}^0 }}
  (-1)^{\ell(v)+r-\ell(w)} g_{w}^{(k)}\,.
\end{equation}
By \cite[Corollary 7.18]{lss},  $g_{s_{r-1} \cdots s_0}^{(k)} = h_r$.
Thus iterating \eqref{eq:horiz-pieri} yields an expression for any  $h_\mu$
($\mu\in \Par^k$)
as a linear combination of $g_\lambda^{(k)}$\,'s.
As $\{h_\mu \mid \mu \in \Par^k, \, |\mu| \le d \}$ forms a basis for the degree  $\le d$
subspace of  $\sym_{(k)}$, the transition matrix from this set to
$\{g_\lambda^{(k)} \mid \lambda \in \Par^k, \, |\lambda| \le d \}$ is invertible, so
\eqref{eq:horiz-pieri} uniquely defines the
$g_\lambda^{(k)}$\,'s.

Now by~\cite{morse}*{\S8}, there is an involution
\(\Omega \from \sym_{(k)} \to \sym_{(k)}\) defined by
\(\Omega(h_r) = g_{1^r}\), and
$\Omega(g_v^{(k)}) =  g_{\tau(v)}^{(k)}$
for all  $v \in \eS_{k+1}^0$, where
$\tau \colon \eS_{k+1} \to \eS_{k+1}$ is the automorphism given by $s_i \mapsto s_{k+1-i}$.
Applying  $\Omega$ to \eqref{eq:horiz-pieri} thus gives \eqref{gpieri}.
Since  $\Omega$ is an involution, it follows that \eqref{gpieri} also uniquely defines the
the  $g_\lambda^{(k)}$\,'s.
\end{proof}

\end{document}
